\documentclass[12pt]{article}
\usepackage{amssymb,amsmath,amsthm,cite,enumerate,geometry,float}
\newtheorem{theorem}{Theorem}

\newtheorem{conjecture}[theorem]{Conjecture}

\newtheorem{claim}{Claim}

\usepackage{lineno}
\usepackage{setspace}
\allowdisplaybreaks

\begin{document}
\onehalfspace

\title{Cyclic Neighborhoods in Digraphs}
\author{Thilo Hartel \and Dieter Rautenbach}
\date{}

\maketitle
\vspace{-10mm}
\begin{center}
Institute of Optimization and Operations Research,\\ Ulm University, Ulm, Germany\\
\texttt{$\{$thilo.hartel,dieter.rautenbach$\}$@uni-ulm.de}
\end{center}

\begin{abstract}
Let $D$ be a digraph of order $n$ and size $m$ 
with the property that 
no out-neighborhood or in-neighborhood of any vertex is acyclic.
We show that, 
if $D$ is strongly connected, then $m\geq 7n/3$, and, 
if $D$ is strongly $2$-connected, then $m\geq 8n/3$.
Both results are best-possible.\\[3mm]
\end{abstract}

\section{Introduction}

The motivation for the research presented here comes from the study of sparse vertex cuts in undirected graphs. 
Chen and Yu \cite{chyu} gave an elegant proof 
for the seminal result 
that every graph of order $n$ 
with less than $2n-3$ edges 
has a vertex cut
that is an independent set.
The existence of vertex cuts
satisfying sparsity conditions involving 
the maximum degree \cite{beraraso},
the degeneracy \cite{bonesovoruvo,
bocofefigosa,chrara, harara} and 
the chromatic number \cite{aubobofopi}
has been studied.
Unlike for the result of Chen and Yu,
for which the extremal graphs are known \cite{lepf,rara},
there are few tight results and 
many open problems in this context.
Restricting the attention to minimum vertex cuts
satisfying sparsity conditions
led to some best-possible results \cite{chtazh,harara}.
For a graph $G$,
an obvious necessary condition for the absence     
of a vertex cut satisfying a sparsity condition
is that the neighborhood $N_G(u)$ 
of every non-universal vertex $u$ in $G$
does not satisfy that sparsity condition.
Motivated by an open problem related to forest cuts,
that is, vertex cuts that are $1$-degenerate,
Chernyshev et al.~\cite{chrara} 
asked for the minimum size of $3$-connected graphs of a given order with the property that every neighborhood of a vertex contains a cycle.
Complete answers and elegant proofs were provided in \cite{litazh,scue}.
Kriesell \cite{kr} suggested to consider similar problems for digraphs. 
In particular, he suggested 
to study
the existence of separators 
in digraphs
that induce acyclic subdigraphs.

Before we explain our results,
we collect some notation,
for which we generally follow \cite{bagu}.
A digraph is {\it cyclic} if it contains 
a directed cycle, and {\it acyclic} otherwise.
A set $U$ of vertices of a digraph $D$ 
is {\it (a)cyclic (in $D$)}
if the subdigraph $D[U]$ of $D$ induced by $U$
is (a)cyclic.
A digraph $D$ is {\it neighborhood-cyclic} if,
for every vertex $u$ of $D$, 
the out-neighborhood $N^+_D(u)$ of $u$ in $D$ and 
the in-neighborhood $N^-_D(u)$ of $u$ in $D$ 
are both cyclic in $D$.
For a vertex $u$ in a digraph $D$
and a set $U$ of vertices of $D$,
let $m_D(u,U)$ be the number of arcs $e$ of $D$
between $u$ and vertices in $U$,
regardless of their orientation.
If there is only one arc $e$ between two vertices $u$ and $v$ in a digraph $D$,
then $e$ is called a {\it single arc} and 
we say that $u$ and $v$ are {\it joined by a single arc}.
If $D$ contains both possible arcs 
of opposite orientation
between $u$ and $v$,
then the pair of these arcs is called a {\it double arc} and
we say that $u$ and $v$ are {\it joined by a double arc}.
For a non-negative integer $k$,
let $[k]$ be the set of positive integers at most $k$, and let $[k]_0=[k]\cup\{ 0\}$.

For an integer $n$ at least $3$,
let $\vec{C}_n^*$ be the digraph that arises from
the disjoint union of 
an isolated vertex $u$ and 
a directed cycle $\vec{C}_{n-1}$ 
by adding all possible $2(n-1)$ arc between $u$ and $C$.
The digraph $\vec{C}_n^*$ has order $n$ and size $3(n-1)$, 
is strongly $2$-connected and neighborhood-cyclic,
and has no acyclic separator.
We believe that $\vec{C}_n^*$ is extremal 
with these properties 
and pose the following.

\begin{conjecture}\label{conjecture1}
If $D$ is a strongly connected digraph
with $n\geq 3$ vertices and $m$ arcs
that does not have an acyclic separator,
then $m\geq 3n-3$.
\end{conjecture}
As in the undirected case,
the study of digraphs 
that do not have acyclic separators
naturally leads to 
consider digraphs that are 
neighborhood-cyclic.
If $D$ is a neighborhood-cyclic digraph of order $n$ and size $m$, 
then $\delta^+(D),\delta^-(D)\geq 2$,
which implies $m\geq 2n$.
The disjoint union of 
complete digraphs $\overset\leftrightarrow{K}_3$ of order $3$
shows that this trivial estimate is best-possible.
Our results in this paper
are the two related best-possible theorems below.

Let ${\cal D}$ be the set of digraphs $D$ 
that arise from $k\geq 2$ disjoint copies $D_0,\ldots,D_{k-1}$ 
of the complete digraph 
$\overset\leftrightarrow{K}_3$
of order $3$ 
by adding, for every $i$ in $[k-1]_0$,
an arc from some vertex in $D_i$ 
to some vertex in $D_{(i+1)\mod k}$.
See Figure \ref{fig0} for an illustration.

\begin{figure}[h]
\centering
\unitlength 1mm 
\linethickness{0.4pt}
\ifx\plotpoint\undefined\newsavebox{\plotpoint}\fi 
\begin{picture}(51.5,30.5)(0,0)
\put(6,0){\circle*{1}}
\put(26,0){\circle*{1}}
\put(26,30){\circle*{1}}
\put(6,30){\circle*{1}}
\put(46,0){\circle*{1}}
\put(46,30){\circle*{1}}
\put(1,10){\circle*{1}}
\put(21,10){\circle*{1}}
\put(21,20){\circle*{1}}
\put(1,20){\circle*{1}}
\put(41,10){\circle*{1}}
\put(41,20){\circle*{1}}
\put(11,10){\circle*{1}}
\put(31,10){\circle*{1}}
\put(31,20){\circle*{1}}
\put(11,20){\circle*{1}}
\put(51,10){\circle*{1}}
\put(51,20){\circle*{1}}
\put(3.5,5){\vector(-1,2){.07}}\multiput(6,0)(-.033557047,.067114094){149}{\line(0,1){.067114094}}
\put(23.5,5){\vector(-1,2){.07}}\multiput(26,0)(-.033557047,.067114094){149}{\line(0,1){.067114094}}
\put(23.5,25){\vector(-1,-2){.07}}\multiput(26,30)(-.033557047,-.067114094){149}{\line(0,-1){.067114094}}
\put(3.5,25){\vector(-1,-2){.07}}\multiput(6,30)(-.033557047,-.067114094){149}{\line(0,-1){.067114094}}
\put(43.5,5){\vector(-1,2){.07}}\multiput(46,0)(-.033557047,.067114094){149}{\line(0,1){.067114094}}
\put(43.5,25){\vector(-1,-2){.07}}\multiput(46,30)(-.033557047,-.067114094){149}{\line(0,-1){.067114094}}
\put(8.5,5){\vector(1,2){.07}}\multiput(6,0)(.033557047,.067114094){149}{\line(0,1){.067114094}}
\put(28.5,5){\vector(1,2){.07}}\multiput(26,0)(.033557047,.067114094){149}{\line(0,1){.067114094}}
\put(28.5,25){\vector(1,-2){.07}}\multiput(26,30)(.033557047,-.067114094){149}{\line(0,-1){.067114094}}
\put(8.5,25){\vector(1,-2){.07}}\multiput(6,30)(.033557047,-.067114094){149}{\line(0,-1){.067114094}}
\put(48.5,5){\vector(1,2){.07}}\multiput(46,0)(.033557047,.067114094){149}{\line(0,1){.067114094}}
\put(48.5,25){\vector(1,-2){.07}}\multiput(46,30)(.033557047,-.067114094){149}{\line(0,-1){.067114094}}
\put(6,10){\vector(1,0){.07}}\put(1,10){\line(1,0){10}}
\put(26,10){\vector(1,0){.07}}\put(21,10){\line(1,0){10}}
\put(26,20){\vector(1,0){.07}}\put(21,20){\line(1,0){10}}
\put(6,20){\vector(1,0){.07}}\put(1,20){\line(1,0){10}}
\put(46,10){\vector(1,0){.07}}\put(41,10){\line(1,0){10}}
\put(46,20){\vector(1,0){.07}}\put(41,20){\line(1,0){10}}
\put(6,10){\vector(-1,0){.07}}\put(11,10){\line(-1,0){10}}
\put(26,10){\vector(-1,0){.07}}\put(31,10){\line(-1,0){10}}
\put(26,20){\vector(-1,0){.07}}\put(31,20){\line(-1,0){10}}
\put(6,20){\vector(-1,0){.07}}\put(11,20){\line(-1,0){10}}
\put(46,10){\vector(-1,0){.07}}\put(51,10){\line(-1,0){10}}
\put(46,20){\vector(-1,0){.07}}\put(51,20){\line(-1,0){10}}
\put(3.5,5){\vector(1,-2){.07}}\multiput(1,10)(.033557047,-.067114094){149}{\line(0,-1){.067114094}}
\put(23.5,5){\vector(1,-2){.07}}\multiput(21,10)(.033557047,-.067114094){149}{\line(0,-1){.067114094}}
\put(23.5,25){\vector(1,2){.07}}\multiput(21,20)(.033557047,.067114094){149}{\line(0,1){.067114094}}
\put(3.5,25){\vector(1,2){.07}}\multiput(1,20)(.033557047,.067114094){149}{\line(0,1){.067114094}}
\put(43.5,5){\vector(1,-2){.07}}\multiput(41,10)(.033557047,-.067114094){149}{\line(0,-1){.067114094}}
\put(43.5,25){\vector(1,2){.07}}\multiput(41,20)(.033557047,.067114094){149}{\line(0,1){.067114094}}
\put(8.5,5){\vector(-1,-2){.07}}\multiput(11,10)(-.033557047,-.067114094){149}{\line(0,-1){.067114094}}
\put(28.5,5){\vector(-1,-2){.07}}\multiput(31,10)(-.033557047,-.067114094){149}{\line(0,-1){.067114094}}
\put(28.5,25){\vector(-1,2){.07}}\multiput(31,20)(-.033557047,.067114094){149}{\line(0,1){.067114094}}
\put(8.5,25){\vector(-1,2){.07}}\multiput(11,20)(-.033557047,.067114094){149}{\line(0,1){.067114094}}
\put(48.5,5){\vector(-1,-2){.07}}\multiput(51,10)(-.033557047,-.067114094){149}{\line(0,-1){.067114094}}
\put(48.5,25){\vector(-1,2){.07}}\multiput(51,20)(-.033557047,.067114094){149}{\line(0,1){.067114094}}
\put(11,15){\vector(0,-1){.07}}\put(11,20){\line(0,-1){10}}
\put(16,10){\vector(1,0){.07}}\put(11,10){\line(1,0){10}}
\put(33.5,5){\vector(3,2){.07}}\multiput(26,0)(.0505050505,.0336700337){297}{\line(1,0){.0505050505}}
\put(46,15){\vector(1,1){.07}}\multiput(41,10)(.0336700337,.0336700337){297}{\line(0,1){.0336700337}}
\put(33.5,25){\vector(-3,2){.07}}\multiput(41,20)(-.0505050505,.0336700337){297}{\line(-1,0){.0505050505}}
\put(18.5,25){\vector(-3,-2){.07}}\multiput(26,30)(-.0505050505,-.0336700337){297}{\line(-1,0){.0505050505}}
\end{picture}
\caption{A digraph of order $18$ in ${\cal D}$.
Double arcs are depicted as 
\unitlength 0.7mm 
\linethickness{0.4pt}
\ifx\plotpoint\undefined\newsavebox{\plotpoint}\fi 
\begin{picture}(10.5,4)(0,0)
\put(0,1.5){\circle*{1}}
\put(10,1.5){\circle*{1}}
\put(5,1.5){\vector(1,0){.07}}
\put(0,1.5){\line(1,0){10}}
\put(5,1.5){\vector(-1,0){.07}}
\end{picture}.}
\label{fig0}
\end{figure}

\begin{theorem}\label{theorem1}
If $D$ is a strongly connected and neighborhood-cyclic digraph 
with $n\geq 5$ vertices and $m$ arcs, then
\begin{eqnarray}\label{e1}
m\geq \frac{7}{3}n
\end{eqnarray}
with equality in \eqref{e1}
if and only if $D\in {\cal D}$.
\end{theorem}
If $D$ is in ${\cal D}$ 
and $D_0,\ldots,D_{k-1}$
are as above, then, 
for every $i$ in $[k-1]_0$,
there is a vertex $v_i$ in $V(D_i)$
with $d_D(v_i)=4$, and
adding the arcs of the directed cycle
$v_0v_1\ldots v_{k-1}v_0$ to $D$
yields a digraph showing that 
also our second result is best-possibe.

\begin{theorem}\label{theorem2}
If $D$ is a strongly $2$-connected 
and neighborhood-cyclic digraph 
with $n\geq 9$ vertices and $m$ arcs,
then 
$$m\geq \frac{8}{3}n.$$
\end{theorem}
The entire rest of this paper is devoted
to the proofs of the two theorems.

\section{Proofs}

\begin{proof}[Proof of Theorem \ref{theorem1}]
Let $D$ be as in the statement.
For a non-negative integer $d$, 
let $V_d=\{ u\in V(D):d_D(u)=d_D^+(u)+d_D^-(u)=d\}$
and let $V_{\geq d}=\{ u\in V(D):d_D(u)\geq d\}$.
Since $D$ is neighborhood-cyclic,
we have $\delta^+(D),\delta^-(D)\geq 2$.
This implies that $V(D)=V_{\geq 4}$,
and that, for every vertex $u$ in $V_4$,
the two out-neighbors of $u$ as well as 
the two in-neighbors of $u$ form $2$-cycles 
$\overset\leftrightarrow{K}_2$.
If some vertex $u$ in $V_4$ 
has both its out-neighbors in $V_4$, 
then $\{ u\}\cup N_D^+(u)$ 
induces a complete digraph 
$\overset\leftrightarrow{K}_3$, 
which contradicts $n\geq 5$.
By symmetry, it follows that
$\Delta^+(D[V_4]),\Delta^-(D[V_4])\leq 1$,
which implies that each component of $D[V_4]$ 
is an isolated vertex,
or a directed cycle, 
or a directed path.

Consider the following discharging procedure:\\[-8mm]
\begin{quote}
{\it Assign to each vertex $u$ of $D$ 
an initial charge of 
$d_D(u)$.
For every vertex $u$ in $V_{\geq 5}$ 
with $m_D(u,V_4)>0$
proceed as follows:
For every arc $e$ of $D$ between $u$ 
and some vertex $v$ in $V_4$,
regardless of its orientation,
move 
$\frac{\left(d_D(u)-\frac{14}{3}\right)}{m_D(u,V_4)}$
charge from $u$ to $v$.}
\end{quote}
Let $c(u)$ be the final charge
for every vertex $u$ of $D$.
Trivially, 
for every vertex $u$ in $V_{\geq 5}$,
we have $c(u)\geq \frac{14}{3}$.
In order to show \eqref{e1},
we show that $c(u)\geq \frac{14}{3}$
also holds for the vertices $u$ in $V_4$.
In order to charaterize 
the extremal digraphs for \eqref{e1},
we sometimes establish more.

\medskip

Let $u$ be a vertex in $V_4$.

\medskip

\noindent {\bf Case 1} {\it $u$ is an isolated vertex in $D[V_4]$.}

\medskip

\noindent Let $N^+_D(u)=\{ v_1,v_2\}$ and 
$N^-_D(u)=\{ w_1,w_2\}$.
Since $N^+_D(u)$ and $N^-_D(u)$ are cyclic,
these sets are complete,
which implies that 
$m_D(v,V_{\geq 5})\geq 2$
for every vertex $v$ in 
$N_D(u)=N^+_D(u)\cup N^-_D(u)$.

\begin{figure}[h]
\centering\hfill
\unitlength 1mm 
\linethickness{0.4pt}
\ifx\plotpoint\undefined\newsavebox{\plotpoint}\fi 
\begin{picture}(35,16)(0,0)
\put(20,3){\circle*{1}}
\put(20,0){\makebox(0,0)[cc]{$u$}}
\put(15,13){\circle*{1}}
\put(25,13){\circle*{1}}
\put(17.5,8){\vector(-1,2){.07}}\multiput(20,3)(-.033557047,.067114094){149}{\line(0,1){.067114094}}
\put(22.5,8){\vector(1,2){.07}}\multiput(20,3)(.033557047,.067114094){149}{\line(0,1){.067114094}}
\put(20,13){\vector(1,0){.07}}\put(15,13){\line(1,0){10}}
\put(20,13){\vector(-1,0){.07}}\put(25,13){\line(-1,0){10}}
\put(10,13){\vector(1,0){.07}}\put(5,13){\line(1,0){10}}
\put(30,13){\vector(1,0){.07}}\put(25,13){\line(1,0){10}}
\put(15,16){\makebox(0,0)[cc]{$v_1$}}
\put(25,16){\makebox(0,0)[cc]{$v_2$}}
\put(0,3){\makebox(0,0)[cc]{$V_4$}}
\put(0,13){\makebox(0,0)[cc]{$V_{\geq 5}$}}
\put(0,8){\line(1,0){35}}
\put(17.5,8){\vector(1,-2){.07}}\multiput(15,13)(.033557047,-.067114094){149}{\line(0,-1){.067114094}}
\put(22.5,8){\vector(-1,-2){.07}}\multiput(25,13)(-.033557047,-.067114094){149}{\line(0,-1){.067114094}}
\end{picture}
\hfill
\unitlength 1mm 
\linethickness{0.4pt}
\ifx\plotpoint\undefined\newsavebox{\plotpoint}\fi 
\begin{picture}(35,16)(0,0)
\put(20,3){\circle*{1}}
\put(20,0){\makebox(0,0)[cc]{$u$}}
\put(0,3){\makebox(0,0)[cc]{$V_4$}}
\put(0,13){\makebox(0,0)[cc]{$V_{\geq 5}$}}
\put(0,8){\line(1,0){35}}
\put(20,13){\circle*{1}}
\put(10,13){\circle*{1}}
\put(30,13){\circle*{1}}
\put(15,13){\vector(1,0){.07}}\put(10,13){\line(1,0){10}}
\put(15,13){\vector(-1,0){.07}}\put(20,13){\line(-1,0){10}}
\put(25,13){\vector(1,0){.07}}\put(20,13){\line(1,0){10}}
\put(25,13){\vector(-1,0){.07}}\put(30,13){\line(-1,0){10}}
\put(10,16){\makebox(0,0)[cc]{$v_1$}}
\put(20,16){\makebox(0,0)[cc]{$v_2$}}
\put(30,16){\makebox(0,0)[cc]{$w_2$}}
\put(15,8){\vector(-1,1){.07}}\multiput(20,3)(-.0336700337,.0336700337){297}{\line(0,1){.0336700337}}
\put(25,8){\vector(-1,-1){.07}}\multiput(30,13)(-.0336700337,-.0336700337){297}{\line(0,-1){.0336700337}}
\put(20,8){\vector(0,1){.07}}\put(20,3){\line(0,1){10}}
\put(20,8){\vector(0,-1){.07}}\put(20,13){\line(0,-1){10}}
\end{picture}\hfill
\unitlength 1mm 
\linethickness{0.4pt}
\ifx\plotpoint\undefined\newsavebox{\plotpoint}\fi 
\begin{picture}(41,16)(0,0)
\put(25,3){\circle*{1}}
\put(20,0){\makebox(0,0)[cc]{$u$}}
\put(20,13){\circle*{1}}
\put(30,13){\circle*{1}}
\put(22.5,8){\vector(-1,2){.07}}\multiput(25,3)(-.033557047,.067114094){149}{\line(0,1){.067114094}}
\put(0,3){\makebox(0,0)[cc]{$V_4$}}
\put(0,13){\makebox(0,0)[cc]{$V_{\geq 5}$}}
\put(10,13){\circle*{1}}
\put(40,13){\circle*{1}}
\put(0,8){\line(1,0){40}}
\put(17.5,8){\vector(-3,2){.07}}\multiput(25,3)(-.0505050505,.0336700337){297}{\line(-1,0){.0505050505}}
\put(27.5,8){\vector(-1,-2){.07}}\multiput(30,13)(-.033557047,-.067114094){149}{\line(0,-1){.067114094}}
\put(32.5,8){\vector(-3,-2){.07}}\multiput(40,13)(-.0505050505,-.0336700337){297}{\line(-1,0){.0505050505}}
\put(15,13){\vector(1,0){.07}}\put(10,13){\line(1,0){10}}
\put(15,13){\vector(-1,0){.07}}\put(20,13){\line(-1,0){10}}
\put(35,13){\vector(1,0){.07}}\put(30,13){\line(1,0){10}}
\put(35,13){\vector(-1,0){.07}}\put(40,13){\line(-1,0){10}}
\put(10,16){\makebox(0,0)[cc]{$v_1$}}
\put(20,16){\makebox(0,0)[cc]{$v_2$}}
\put(30,16){\makebox(0,0)[cc]{$w_1$}}
\put(40,16){\makebox(0,0)[cc]{$w_2$}}
\end{picture}
\hfill $\mbox{}$
\caption{$u$ is isolated in $D[V_4]$.}
\label{fig1}
\end{figure}
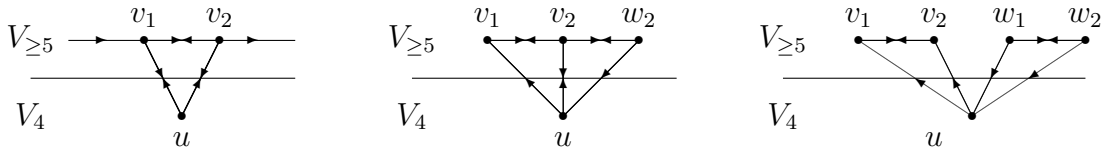
First, suppose that $N^+_D(u)=N^-_D(u)$
as shown in the left of Figure \ref{fig1}.
If $\max\{ d_D(v_1),d_D(v_2)\}\geq 6$,
then one of the two vertices $v_1$ and $v_2$
sends at least $\frac{2}{3}$ charge to $u$
and the other sends a positive amount of charge to $u$, which implies $c(u)>4+\frac{2}{3}=\frac{14}{3}$.
Now, we assume that $d_D(v_1)=d_D(v_2)=5$.
Since $D$ is strongly connected and $n\geq 5$,
we may assume, by symmetry, 
that $v_1$ has an in-neighbor $x_1$ 
distinct from $u$ and $v_2$, 
and 
that $v_2$ has an out-neighbor $x_2$
distinct from $u$ and $v_1$.
Since $D$ is strongly connected,
we have $x_1\not=x_2$.
Since $m_D(x_1,\{ u,v_2\})=0$,
the cyclicity of $N_D^+(x_1)$ 
implies that $x_1\in V_{\geq 5}$.
Similarly, 
we obtain that $x_2\in V_{\geq 5}$.
Now, the vertices $v_1$ and $v_2$ 
both send exactly $\frac{1}{3}$ charge to $u$,
which implies
$c(u)=4+\frac{2}{3}=\frac{14}{3}$.

Next, suppose that $N^+_D(u)\cap N^-_D(u)=\{ v_2\}=\{ w_1\}$
as shown in the middle of Figure \ref{fig1}.
The vertex $v_2$ sends exactly 
$$\frac{2\left(d_D(v_2)-\frac{14}{3}\right)}{m_D(v_2,V_4)}\geq 
\frac{2\left(d_D(v_2)-\frac{14}{3}\right)}{d_D(v_2)-4}\stackrel{d_D(v_2)>5}{>}\frac{2}{3}$$ charge to $u$,
which implies $c(u)>\frac{14}{3}$.

Finally, suppose that $N^+_D(u)\cap N^-_D(u)=\emptyset$
as shown in the right of Figure \ref{fig1}.
If two vertices $v$ in $N_D(u)$ have degree 
$d_D(v)$ at least $6$, 
then each of these sends 
$$\frac{\left(d_D(v)-\frac{14}{3}\right)}{m_D(v,V_4)}\geq 
\frac{\left(d_D(v)-\frac{14}{3}\right)}{d_D(v)-2}\stackrel{d_D(v)\geq 6}{\geq}\frac{1}{3}$$ 
charge to $u$,
and the remaining two vertices in $N_D(u)$
send a positive amount of charge to $u$,
which implies $c(u)>\frac{14}{3}$.
Hence, we may assume that 
at most one vertex in $N_D(u)$
has degree at least $6$.

\begin{claim}\label{claim1}
If $d_D(v_1)=d_D(v_2)=5$,
then $m_D(v_1,V_4)=m_D(v_2,V_4)=1$.

If $d_D(w_1)=d_D(w_2)=5$,
then $m_D(w_1,V_4)=m_D(w_2,V_4)=1$.
\end{claim}
\begin{proof}[Proof of Claim \ref{claim1}]
Suppose that $d_D(v_1)=d_D(v_2)=5$.
Since $v_2$ and $u$ do not form a $2$-cycle,
it follows that 
$d_D^-(v_1)=3$
and 
$d_D^+(v_1)=2$.
This implies the existence of a vertex $x$
with $N_D^+(v_1)=\{ v_2,x\}$
such that $v_2$ and $x$ form a $2$-cycle.
Since $N_D^+(v_2)=\{ v_1,x\}$,
it follows that also $v_1$ and $x$ form a $2$-cycle.
If $x\in V_4$, then $D$ is not strongly connected.
Hence, the arc $(u,v_1)$ is the only arc
between $v_1$ and $V_4$.
By symmetry, the claim follows.
\end{proof}
If exactly one vertex in $N_D(u)$ has degree at least $6$,
then, by Claim \ref{claim1},
one vertex in $N_D(u)$ sends at least $\frac{1}{3}$ charge to $u$
and two other vertices in $N_D(u)$ each send 
exactly $\frac{1}{3}$ charge to $u$,
which implies $c(u)>\frac{14}{3}$.
If no vertex in $N_D(u)$ has degree at least $6$,
then, by Claim \ref{claim1},
each vertex in $N_D(u)$ sends exactly
$\frac{1}{3}$ charge to $u$,
which implies $c(u)>\frac{14}{3}$.

\medskip

\noindent {\bf Case 2} {\it $u$ belongs to a directed cycle 
$C=u_1u_2\ldots u_\ell u_1$
that is a component in $D[V_4]$.}

\medskip

\noindent Let $v$ be the out-neighbor of $u_1$ in $V_{\geq 5}$.
The cyclicity of $N_D^+(u_1)$
implies that $u_2$ and $v$ form a $2$-cycle,
and the cyclicity of $N_D^-(u_2)$ 
implies that $u_1$ and $v$ form a $2$-cycle.
Now, 
if $\ell>2$,
then the cyclicity of $N_D^+(u_2)$ 
implies that $u_3$ and $v$ form a $2$-cycle,
if $\ell>3$,
then the cyclicity of $N_D^+(u_3)$ 
implies that $u_4$ and $v$ form a $2$-cycle,
and so forth.
Altogether, it follows inductively that 
$u_i$ and $v$ form a $2$-cycle
for every $i\in [\ell]$.

See Figure \ref{fig2} for an illustration.

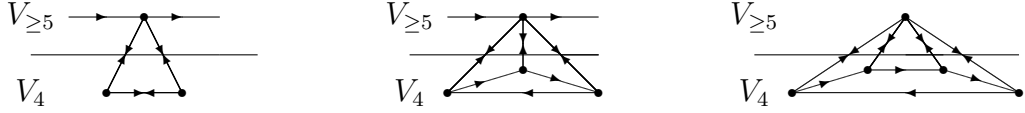
\begin{figure}[h]
\centering\hfill 
\unitlength 1mm 
\linethickness{0.4pt}
\ifx\plotpoint\undefined\newsavebox{\plotpoint}\fi 
\begin{picture}(30,13.5)(0,0)
\put(15,13){\circle*{1}}
\put(10,3){\circle*{1}}
\put(20,3){\circle*{1}}
\put(12.5,8){\vector(-1,-2){.07}}\multiput(15,13)(-.033557047,-.067114094){149}{\line(0,-1){.067114094}}
\put(17.5,8){\vector(1,-2){.07}}\multiput(15,13)(.033557047,-.067114094){149}{\line(0,-1){.067114094}}
\put(15,3){\vector(1,0){.07}}\put(10,3){\line(1,0){10}}
\put(15,3){\vector(-1,0){.07}}\put(20,3){\line(-1,0){10}}
\put(10,13){\vector(1,0){.07}}\put(5,13){\line(1,0){10}}
\put(20,13){\vector(1,0){.07}}\put(15,13){\line(1,0){10}}
\put(0,3){\makebox(0,0)[cc]{$V_4$}}
\put(0,13){\makebox(0,0)[cc]{$V_{\geq 5}$}}
\put(12.5,8){\vector(1,2){.07}}\multiput(10,3)(.033557047,.067114094){149}{\line(0,1){.067114094}}
\put(17.5,8){\vector(-1,2){.07}}\multiput(20,3)(-.033557047,.067114094){149}{\line(0,1){.067114094}}
\put(0,8){\line(1,0){25}}
\put(25,8){\line(1,0){5}}
\put(30,8){\line(0,1){0}}
\end{picture}
\hfill 
\unitlength 1mm 
\linethickness{0.4pt}
\ifx\plotpoint\undefined\newsavebox{\plotpoint}\fi 
\begin{picture}(25.5,13.5)(0,0)
\put(15,13){\circle*{1}}
\put(5,3){\circle*{1}}
\put(15,6){\circle*{1}}
\put(25,3){\circle*{1}}
\put(10,13){\vector(1,0){.07}}\put(5,13){\line(1,0){10}}
\put(20,13){\vector(1,0){.07}}\put(15,13){\line(1,0){10}}
\put(0,3){\makebox(0,0)[cc]{$V_4$}}
\put(0,13){\makebox(0,0)[cc]{$V_{\geq 5}$}}
\put(0,8){\line(1,0){25}}
\put(20,8){\line(1,0){5}}
\put(25,8){\line(0,1){0}}
\put(15,9.5){\vector(0,1){.07}}\put(15,6){\line(0,1){7}}
\put(15,9.5){\vector(0,-1){.07}}\put(15,13){\line(0,-1){7}}
\put(10,8){\vector(1,1){.07}}\multiput(5,3)(.0336700337,.0336700337){297}{\line(0,1){.0336700337}}
\put(10,8){\vector(-1,-1){.07}}\multiput(15,13)(-.0336700337,-.0336700337){297}{\line(0,-1){.0336700337}}
\put(15,3){\vector(-1,0){.07}}\put(25,3){\line(-1,0){20}}
\put(20,8){\vector(-1,1){.07}}\multiput(25,3)(-.0336700337,.0336700337){297}{\line(0,1){.0336700337}}
\put(20,8){\vector(1,-1){.07}}\multiput(15,13)(.0336700337,-.0336700337){297}{\line(0,-1){.0336700337}}
\put(10,4.5){\vector(3,1){.07}}\multiput(5,3)(.11235955,.03370787){89}{\line(1,0){.11235955}}
\put(20,4.5){\vector(3,-1){.07}}\multiput(15,6)(.11235955,-.03370787){89}{\line(1,0){.11235955}}
\end{picture}
\hfill 
\unitlength 1mm 
\linethickness{0.4pt}
\ifx\plotpoint\undefined\newsavebox{\plotpoint}\fi 
\begin{picture}(35.5,13.5)(0,0)
\put(20,13){\circle*{1}}
\put(5,3){\circle*{1}}
\put(15,6){\circle*{1}}
\put(25,6){\circle*{1}}
\put(35,3){\circle*{1}}
\put(0,3){\makebox(0,0)[cc]{$V_4$}}
\put(0,13){\makebox(0,0)[cc]{$V_{\geq 5}$}}
\put(0,8){\line(1,0){35}}
\put(20,8){\line(1,0){5}}
\put(25,8){\line(0,1){0}}
\put(12.5,8){\vector(3,2){.07}}\multiput(5,3)(.0505050505,.0336700337){297}{\line(1,0){.0505050505}}
\put(12.5,8){\vector(-3,-2){.07}}\multiput(20,13)(-.0505050505,-.0336700337){297}{\line(-1,0){.0505050505}}
\put(20,3){\vector(-1,0){.07}}\put(35,3){\line(-1,0){30}}
\put(30,4.5){\vector(3,-1){.07}}\multiput(25,6)(.11235955,-.03370787){89}{\line(1,0){.11235955}}
\put(27.5,8){\vector(-3,2){.07}}\multiput(35,3)(-.0505050505,.0336700337){297}{\line(-1,0){.0505050505}}
\put(27.5,8){\vector(3,-2){.07}}\multiput(20,13)(.0505050505,-.0336700337){297}{\line(1,0){.0505050505}}
\put(20,6){\vector(1,0){.07}}\put(15,6){\line(1,0){10}}
\put(10,4.5){\vector(3,1){.07}}\multiput(5,3)(.11235955,.03370787){89}{\line(1,0){.11235955}}
\put(17.5,9.5){\vector(3,4){.07}}\multiput(15,6)(.033557047,.046979866){149}{\line(0,1){.046979866}}
\put(17.5,9.5){\vector(-3,-4){.07}}\multiput(20,13)(-.033557047,-.046979866){149}{\line(0,-1){.046979866}}
\put(22.5,9.5){\vector(-3,4){.07}}\multiput(25,6)(-.033557047,.046979866){149}{\line(0,1){.046979866}}
\put(22.5,9.5){\vector(3,-4){.07}}\multiput(20,13)(.033557047,-.046979866){149}{\line(0,-1){.046979866}}
\end{picture}
\hfill $\mbox{}$
\caption{Directed cycle components of $D[V_4]$ of lengths $2$, $3$, and $4$.}
\label{fig2}
\end{figure}

If $\ell\leq 3$, 
then the strong connectivity implies
that $v$ has an in-neighbor $x_1$
as well as an out-neighbor $x_2$
with $x_1,x_2\not\in V(C)$.
If $\ell=2$, then the cyclicity of $N_D^+(x_1)$ implies that 
$x_1\in V_{\geq 5}$ or $d_D^-(v)\geq 4$, 
and the cyclicity of $N_D^-(x_2)$ implies that 
$x_2\in V_{\geq 5}$ or $d_D^+(v)\geq 4$.
Note that $d_D(v)\geq 6$.
Furthermore, 
$d_D(v)=6$ implies that $m_D(v,V_4)\leq d_D(v)-2$, and 
$d_D(v)=7$ implies that $m_D(v,V_4)\leq d_D(v)-1$.
It follows that $v$ sends 
$$
\frac{2\left(d_D(v)-\frac{14}{3}\right)}{m_D(v,V_4)}\geq 
\frac{2}{3}$$ 
charge to $u$.
More precisely, it follows that $c(u)\geq \frac{14}{3}$
with equality if and only if 
$m_D(v,V_4)=d_D(v)-2$ and $d_D(v)=6$.
If $\ell\geq 3$, then $d_D(v)\geq 8$,
and $v$ sends 
$$
\frac{2\left(d_D(v)-\frac{14}{3}\right)}{m_D(v,V_4)}\geq 
\frac{2\left(d_D(v)-\frac{14}{3}\right)}{d_D(v)}\stackrel{d_D(v)\geq 8}{>}
\frac{2}{3}$$ 
charge to $u$,
which implies $c(u)>\frac{14}{3}$.

\medskip

\noindent {\bf Case 3} {\it $u$ belongs to a directed path 
$P=u_1u_2\ldots u_\ell$ with $\ell\geq 2$
that is a component in $D[V_4]$.}

\medskip

\noindent Let $v$ be the out-neighbor of $u_1$ in $V_{\geq 5}$.
As in Case 2, it follows that 
$u_i$ and $v$ form a $2$-cycle
for every $i\in [\ell]$.
Let $w$ be the in-neighbor of $u_1$ 
distinct from $v$.
In view of the case assumption,
we have that $w\in V_{\geq 5}$.
The cyclicity of $N_D^-(u_1)$
implies that $v$ and $w$ form a $2$-cycle.

See Figure \ref{fig3} for an illustration.

\begin{figure}[h]
\centering\hfill 
\unitlength 1mm 
\linethickness{0.4pt}
\ifx\plotpoint\undefined\newsavebox{\plotpoint}\fi 
\begin{picture}(30,15)(0,0)
\put(20,13){\circle*{1}}
\put(10,13){\circle*{1}}
\put(15,3){\circle*{1}}
\put(25,3){\circle*{1}}
\put(17.5,8){\vector(-1,-2){.07}}\multiput(20,13)(-.033557047,-.067114094){149}{\line(0,-1){.067114094}}
\put(22.5,8){\vector(1,-2){.07}}\multiput(20,13)(.033557047,-.067114094){149}{\line(0,-1){.067114094}}
\put(15,13){\vector(1,0){.07}}\put(10,13){\line(1,0){10}}
\put(0,3){\makebox(0,0)[cc]{$V_4$}}
\put(0,13){\makebox(0,0)[cc]{$V_{\geq 5}$}}
\put(17.5,8){\vector(1,2){.07}}\multiput(15,3)(.033557047,.067114094){149}{\line(0,1){.067114094}}
\put(22.5,8){\vector(-1,2){.07}}\multiput(25,3)(-.033557047,.067114094){149}{\line(0,1){.067114094}}
\put(0,8){\line(1,0){25}}
\put(30,8){\line(0,1){0}}
\put(20,3){\vector(1,0){.07}}\put(15,3){\line(1,0){10}}
\put(12.5,8){\vector(1,-2){.07}}\multiput(10,13)(.033557047,-.067114094){149}{\line(0,-1){.067114094}}
\put(15,13){\vector(-1,0){.07}}\put(20,13){\line(-1,0){10}}
\put(15,0){\makebox(0,0)[cc]{$u_1$}}
\put(20,15){\makebox(0,0)[cc]{$v$}}
\put(10,15){\makebox(0,0)[cc]{$w$}}
\put(24,8){\line(1,0){3}}
\end{picture}
\hfill 
\unitlength 1mm 
\linethickness{0.4pt}
\ifx\plotpoint\undefined\newsavebox{\plotpoint}\fi 
\begin{picture}(27,15)(0,0)
\put(15,13){\circle*{1}}
\put(5,13){\circle*{1}}
\put(5,3){\circle*{1}}
\put(15,3){\circle*{1}}
\put(25,3){\circle*{1}}
\put(10,13){\vector(1,0){.07}}\put(5,13){\line(1,0){10}}
\put(0,3){\makebox(0,0)[cc]{$V_4$}}
\put(0,13){\makebox(0,0)[cc]{$V_{\geq 5}$}}
\put(22,8){\line(0,1){0}}
\put(10,3){\vector(1,0){.07}}\put(5,3){\line(1,0){10}}
\put(20,3){\vector(1,0){.07}}\put(15,3){\line(1,0){10}}
\put(10,13){\vector(-1,0){.07}}\put(15,13){\line(-1,0){10}}
\put(5,0){\makebox(0,0)[cc]{$u_1$}}
\put(15,15){\makebox(0,0)[cc]{$v$}}
\put(5,15){\makebox(0,0)[cc]{$w$}}
\put(5,8){\vector(0,-1){.07}}\put(5,13){\line(0,-1){10}}
\put(10,8){\vector(1,1){.07}}\multiput(5,3)(.0336700337,.0336700337){297}{\line(0,1){.0336700337}}
\put(10,8){\vector(-1,-1){.07}}\multiput(15,13)(-.0336700337,-.0336700337){297}{\line(0,-1){.0336700337}}
\put(15,8){\vector(0,-1){.07}}\put(15,13){\line(0,-1){10}}
\put(15,8){\vector(0,1){.07}}\put(15,3){\line(0,1){10}}
\put(20,8){\vector(1,-1){.07}}\multiput(15,13)(.0336700337,-.0336700337){297}{\line(0,-1){.0336700337}}
\put(20,8){\vector(-1,1){.07}}\multiput(25,3)(-.0336700337,.0336700337){297}{\line(0,1){.0336700337}}
\put(0,8){\line(1,0){27}}
\end{picture}
\hfill $\mbox{}$
\caption{Directed path components of $D[V_4]$ of lengths $1$ and $2$.}
\label{fig3}
\end{figure}
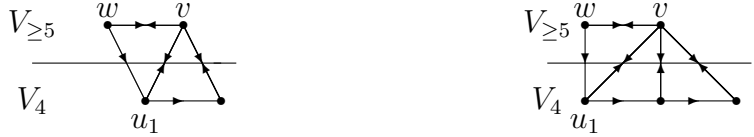

Since $d_D(v)\geq 2\ell+2$,
the vertex $v$ sends 
$\frac{2}{3}$ charge to each vertex on $C$,
and the vertex $w$ sends a positive amount of charge
to $u_1$.
This implies that 
$c(u_1)>\frac{14}{3}$
and 
$c(u_i)\geq \frac{14}{3}$
for every $i\in [\ell]\setminus \{ 1\}$.

\medskip

\noindent In all three cases,
we obtain $c(u)\geq \frac{14}{3}$
for every vertex $u$ of $D$,
which implies \eqref{e1}.
Clearly, the digraphs in ${\cal D}$
have the desired properties and 
satisfy \eqref{e1} with equality.
Now, suppose that \eqref{e1} 
is satisfied with equality.
In this case,
the final charge $c(u)$
is exactly $\frac{14}{3}$
for every vertex $u$ of $D$.
This implies that 
$m_D(v,V_4)>0$ for every vertex $v$ in $V_{\geq 5}$.
By the discussion in the three cases,
each $u$ in $V_4$ 
\begin{itemize}
\item either induces a
$\overset\leftrightarrow{K}_3$
with two vertices $v_1$ and $v_2$ of degree $5$
such that 
$v_1$ has an in-neighbor in $V_{\geq 5}$
and  
$v_2$ has an out-neighbor in $V_{\geq 5}$,
\item or induces a
$\overset\leftrightarrow{K}_3$
with one vertex $u'$ of degree $4$
and one vertex $v$ of degree $6$
such that $v$ has an out-neighbor
as well as an in-neighbor in $V_{\geq 5}$.
\end{itemize}
Since $D$ is strongly-connected
and $n\geq 5$, this implies that $D\in {\cal D}$,
which completes the proof.
\end{proof}

\setcounter{claim}{0}

\begin{proof}[Proof of Theorem \ref{theorem2}]
Let $D$ be as in the statement.
Let $V_d$ and $V_{\geq d}$ be as in the proof of Theorem \ref{theorem1}.
Again, we obtain that $\delta^+(D),\delta^-(D)\geq 2$ and that $V(D)=V_{\geq 4}$.
In view of the desired statement, we may assume, for a contradiction,
that $m<\frac{8}{3}n$.

\begin{claim}\label{claim2}
If $d_D^\sigma(v)=2$ for some vertex $v$ of $D$
and some $\sigma\in \{ +,-\}$,
then $N_D^\sigma(v)\cap V_{\geq 6}\not=\emptyset$.
\end{claim}
\begin{proof}[Proof of Claim \ref{claim2}]
By symmetry, we may assume that $N_D^+(v)=\{ w_1,w_2\}$.
Suppose, for a contradiction, that $d_D(w_1),d_D(w_2)\leq 5$.
Since $N_D^+(v)$ is cyclic, 
there is a double arc between $w_1$ and $w_2$.
First, suppose that $w_1,w_2\not\in N_D^-(v)$.
Since there is no double arc between the two in-neighbors $v$ and $w_2$ of $w_1$,
it follows that $d_D^-(w_1)=3$ and $d_D^+(w_1)=2$.
It follows that 
there is a vertex $x$ with $N_D^+(w_1)=\{ w_2,x\}$ 
such that $w_2$ and $x$ are joined by a double arc.
Now, we have $N_D^+(w_2)=\{ w_1,x\}$, 
which implies that $w_1$ and $x$ are joined by a double arc.
In $D-x$ no arc leaves $\{ v,w_1,w_2\}$,
which is a contradiction.
Next, suppose that $w_2\in N_D^-(v)$, that is,
there is a double arc between $v$ and $w_2$.
Since $N_D^+(w_2)$ or $N_D^-(w_2)$ equals $\{ v,w_1\}$,
there is a double arc between $v$ and $w_1$.
Now, the set
$\{ v,w_1,w_2\}$ induces $\overset\leftrightarrow{K}_3$, and,
since $d_D(v),d_D(w_1),d_D(w_2)\leq 5$,
either at most one arc leaves $\{ v,w_1,w_2\}$ in $D$
or at most one arc enters $\{ v,w_1,w_2\}$ in $D$.
This implies that $D$ is not strongly $2$-connected,
which is a contradiction.
It follows that $N_D^+(v)$ contains a vertex $w$ with $d_D(w)\geq 6$, which completes the proof of the claim.
\end{proof}

\begin{claim}\label{claim3}
Each component of $D[V_4]$ is an isolated vertex or a directed path.
If the directed path $P=u_0u_1\ldots u_\ell$ with $\ell\geq 1$ 
is a component of $D[V_4]$, 
then there are vertices $v$, $x$, and $y$ in $V_{\geq 5}$
such that 
$V(P)\cup \{ x,y\}\subseteq N_D^+(v)\cap N_D^-(v)$,
$u_0\in N_D^+(x)$, and 
$u_\ell\in N_D^-(y)$.
Note that $x=y$ is possible.
Moreover, we have $d_D(v)\geq 2\ell+6$.
\end{claim}
\begin{proof}[Proof of Claim \ref{claim3}]
As in Theorem \ref{theorem1}, it follows that $\Delta^+(D[V_4]),\Delta^-(D[V_4])\leq 1$,
which implies that each component of $D[V_4]$ is 
an isolated vertex, 
or a directed cycle, 
or a directed path.
Suppose, for a contradiction, 
that some directed cycle $C=u_1\ldots u_\ell u_1$ is a component of $D[V_4]$.
Since $D$ is neighborhood-cyclic,
there is a vertex $v$ such that $V(C)\subseteq N_D^+(v)\cap N_D^-(v)$.
Since in $D-v$ no arc leaves or enters $V(C)$,
we obtain that $V(D)=V(C)\cup \{ v\}$, and
$m=3n-3\geq \frac{8}{3}n$, which is a contradiction.

Now, suppose that some directed path $P=u_0u_1\ldots u_\ell$ with $\ell\geq 1$ 
is a component of $D[V_4]$.
Since $D$ is neighborhood-cyclic,
there is a vertex $v$ such that $V(P)\subseteq N_D^+(v)\cap N_D^-(v)$.
Furthermore, 
if $x$ and $y$ are such that $N_D^-(u_0)=\{ v,x\}$ and $N_D^+(u_\ell)=\{ v,y\}$,
then $\{ x,y\}\subseteq N_D^+(v)\cap N_D^-(v)$.
If $x\neq y$, then $d_D(v)\geq 2\ell +6$.
If $x=y$, then,
since $D-x$ is strongly connected, 
we obtain that $d_D(v)\geq 2\ell+6$, 
which completes the proof of the claim.
\end{proof}

Consider the following discharging procedure:\\[-8mm]

\begin{quote}
{\it 
Assign to each vertex $u$ of $D$ an initial charge of $d_D(u)$.
For every vertex $u$ in $V_4$, 
set ${\rm demand}(u,+)$ and ${\rm demand}(u,-)$ equal to $\frac{2}{3}$.
For every vertex $u$ in $V_5$ and $\sigma\in \{ +,-\}$ such that $d_D^\sigma(u)=2$, 
set ${\rm demand}(u,\sigma)$ equal to $\frac{1}{3}$.
Let 
$$P=\Big\{ (u,\sigma):u\in V_{\leq 5},\sigma\in \{ +,-\},\mbox{ and }d_D^\sigma(u)=2\Big\}$$ 
and declare all pairs in $P$ {\it demanding}.
We now redistribute the charge in four phases.

{\bf Phase 1}\\
As long as there are five distinct vertices $x$, $y$, $u$, $v$, and $w$ 
as well as $\sigma\in \{ +,-\}$ 
such that 
$x,y\in V_4$, 
$u\in V_5$,
$v\in V_6$,
$w\in V_{\geq 7}$,
$\{ u,v,w\}$ induces $\overset\leftrightarrow{K}_3$,
$N_D^\sigma(x)=\{ u,v\}$, 
$N_D^{-\sigma}(y)=\{ v,w\}$, and 
the three pairs $(x,\sigma)$, $(u,\sigma)$, and $(y,-\sigma)$ are demanding,
see Figure \ref{figspecial1} for an illustration,
\begin{figure}[H]
\begin{center}
\unitlength 1mm 
\linethickness{0.4pt}
\ifx\plotpoint\undefined\newsavebox{\plotpoint}\fi 
\begin{picture}(20.5,15)(0,0)
\put(10,13){\circle*{1}}
\put(20,13){\circle*{1}}
\put(0,13){\circle*{1}}
\put(5,3){\circle*{1}}
\put(15,3){\circle*{1}}
\put(10,13){\line(-1,-2){5}}
\put(5,3){\line(1,0){10}}
\put(15,3){\line(-1,2){5}}
\put(10,15){\makebox(0,0)[cc]{$v$}}
\put(5,1){\makebox(0,0)[cc]{$u$}}
\put(15,1){\makebox(0,0)[cc]{$w$}}
\put(0,13){\line(1,0){20}}
\put(20,13){\line(-1,-2){5}}
\put(5,3){\line(-1,2){5}}
\put(0,13){\vector(1,-2){3}}
\put(0,13){\vector(1,0){5}}
\put(10,13){\vector(1,0){5}}
\put(15,3){\vector(1,2){3}}
\put(0,15){\makebox(0,0)[cc]{$x$}}
\put(20,15){\makebox(0,0)[cc]{$y$}}
\put(5,3){\vector(1,2){3}}
\put(8,9){\vector(-1,-2){.07}}\multiput(10,13)(-.03333333,-.06666667){60}{\line(0,-1){.06666667}}
\put(12,9){\vector(1,-2){.07}}\multiput(10,13)(.03333333,-.06666667){60}{\line(0,-1){.06666667}}
\put(15,3){\vector(-1,2){3}}
\put(5,3){\vector(1,0){5}}
\put(15,3){\vector(-1,0){5}}
\end{picture}
\end{center}
\caption{The configuration in Phase 1 with $\sigma=+$.
Reversing all orientations corresponds to $\sigma=-$.}
\label{figspecial1}
\end{figure}
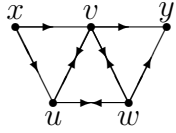
\begin{itemize}
\item $v$ sends ${\rm demand}(x,\sigma)=\frac{2}{3}$ charge to $x$.
\item $w$ sends ${\rm demand}(u,\sigma)=\frac{1}{3}$ charge to $u$.
\item $w$ sends ${\rm demand}(y,-\sigma)=\frac{2}{3}$ charge to $y$.
\item Declare the three pairs $(x,\sigma)$, $(u,\sigma)$, and $(y,-\sigma)$ {\it not demanding}.
\end{itemize}
{\bf Phase 2}\\
As long as there are four distinct vertices $x$, $u$, $v$, and $w$ 
as well as $\sigma\in \{ +,-\}$ 
such that 
$x\in V_4$, 
$u\in V_5$,
$v,w\in V_6$,
there is a single arc between $x$ and $v$,
$\{ u,v,w\}$ induces $\overset\leftrightarrow{K}_3$,
$N_D^\sigma(x)=\{ u,v\}$, and 
the two pairs $(x,\sigma)$ and $(u,\sigma)$ are demanding,
see Figure \ref{figspecial2} for an illustration,
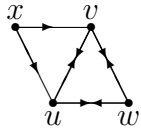
\begin{figure}[H]
\begin{center}
\unitlength 1mm 
\linethickness{0.4pt}
\ifx\plotpoint\undefined\newsavebox{\plotpoint}\fi 
\begin{picture}(15.5,15)(0,0)
\put(10,13){\circle*{1}}
\put(0,13){\circle*{1}}
\put(5,3){\circle*{1}}
\put(15,3){\circle*{1}}
\put(10,13){\line(-1,-2){5}}
\put(5,3){\line(1,0){10}}
\put(15,3){\line(-1,2){5}}
\put(10,15){\makebox(0,0)[cc]{$v$}}
\put(5,1){\makebox(0,0)[cc]{$u$}}
\put(15,1){\makebox(0,0)[cc]{$w$}}
\put(5,3){\line(-1,2){5}}
\put(0,13){\vector(1,-2){3}}
\put(0,13){\vector(1,0){5}}
\put(0,15){\makebox(0,0)[cc]{$x$}}
\put(0,13){\line(1,0){10}}
\put(5,3){\vector(1,2){3}}
\put(8,9){\vector(-1,-2){.07}}\multiput(10,13)(-.03333333,-.06666667){60}{\line(0,-1){.06666667}}
\put(12,9){\vector(1,-2){.07}}\multiput(10,13)(.03333333,-.06666667){60}{\line(0,-1){.06666667}}
\put(15,3){\vector(-1,2){3}}
\put(5,3){\vector(1,0){5}}
\put(15,3){\vector(-1,0){5}}
\end{picture}\end{center}
\caption{The configuration in Phase 2 with  with $\sigma=+$.}
\label{figspecial2}
\end{figure}
\begin{itemize}
\item $v$ sends ${\rm demand}(x,\sigma)=\frac{2}{3}$ charge to $x$.
\item $w$ sends ${\rm demand}(u,\sigma)=\frac{1}{3}$ charge to $u$.
\item Declare the two pairs $(x,\sigma)$ and $(u,\sigma)$ {\it not demanding}.
\end{itemize}
{\bf Phase 3}\\
For every demanding pair $(u,\sigma)$ with $d_D(u)=4$,
Claim \ref{claim2} implies
$N_D^\sigma(u)\cap V_{\geq 6}\not=\emptyset$,
and we proceed as follows:
\begin{itemize}
\item If $N_D^\sigma(u)\cap V_{\geq 6}=\{ v\}$, 
then $v$ sends ${\rm demand}(u,\sigma)=\frac{2}{3}$ charge to $u$.
\item If $N_D^\sigma(u)\cap V_{\geq 6}=\{ v,w\}$, then
\begin{eqnarray*}
\mbox{ $v$ sends }&  \frac{d_D(v)-5}{d_D(v)+d_D(w)-10}\cdot {\rm demand}(u,\sigma) & \mbox{ charge to $u$ and}\\ 
\mbox{ $w$ sends }& \frac{d_D(w)-5}{d_D(v)+d_D(w)-10}\cdot {\rm demand}(u,\sigma) & \mbox{ charge to $u$.}
\end{eqnarray*}
\item Declare the pair $(u,\sigma)$ {\it not demanding}.
\end{itemize}
{\bf Phase 4}\\
For every demanding pair $(u,\sigma)$ with $d_D(u)=5$,
Claim \ref{claim2} implies
$N_D^\sigma(u)\cap V_{\geq 6}\not=\emptyset$,
and we proceed as follows:

Let $N^\sigma_D(u)=\{v,w\}$.
\begin{itemize}
\item If $d_D(v)=d_D(w)$, then 
$v$ and $w$ each send $\frac{1}{2}{\rm demand}(u,\sigma)=\frac{1}{6}$ charge to $u$. 
\item If $d_D(v)>d_D(w)$, then 
$v$ sends ${\rm demand}(u,\sigma)=\frac{1}{3}$ charge to $u$.
\item If $d_D(v)<d_D(w)$, then 
$w$ sends ${\rm demand}(u,\sigma)=\frac{1}{3}$ charge to $u$.
\item Declare the pair $(u,\sigma)$ {\it not demanding}.
\end{itemize}
}
\end{quote}
Whenever a vertex $u$ sends $c\cdot {\rm demand}(v,\sigma)$
charge to a vertex $v$ 
for some $c\in(0,1]$, 
the charge is considered to be sent along 
a unique arc determined by $\sigma$: namely, along $(v,u)$ if $\sigma =+$, and along $(u,v)$ if $\sigma=-$.
For brevity, we simply say that $u$ 
is sending charge along the corresponding arc.
Let $c(u)$ be the final charge of a vertex $u$ in $D$ after all phases of the discharging.
Clearly, at that point every pair in $P$ is declared \emph{not demanding}. 
Note that only vertices in $V_{\geq 6}$
send charge to neighbors.
Moreover, every vertex $u$ in $V_4$ satisfies 
$$c(u)=4+{\rm demand}(u,+) +{\rm demand}(u,-)=4+\frac{2}{3}+\frac{2}{3}=\frac{16}{3}.$$ 
Similarly, for every vertex $u$ in $V_5$ and $\sigma\in \{+,-\}$ such that $d_D^\sigma(u)=2$, we have 
$$c(u)=5+{\rm demand}(u,\sigma)=5+\frac{1}{3}=\frac{16}{3}.$$ 

\begin{claim}\label{claim4}
If $u$ is a vertex in $V_6$, 
then $c(u)\geq \frac{16}{3}$.
\end{claim}
\begin{proof}[Proof of Claim \ref{claim4}]
Since $6-\frac{2}{3}=\frac{16}{3}$,
we may assume that $u$ sends charge 
along at least two arcs,
which implies that there is 
at least one neighbor $v$ of $u$ 
such that $u$ and $v$ 
are joined by a double arc.

\medskip

\noindent {\bf Case 1.} {\it $u$ is incident with exactly one double arc.}

\medskip 

\noindent Let $D$ contain a double arc between $u$ and $v$.
First observe that $u$ sends no charge 
in Phase 1 or Phase 2.
Next, suppose, for a contradiction, 
that $d_D(v)\leq 5$.
Let $\sigma$ in $\{+,-\}$ be such that $d_D^\sigma(v)=2$.
Let $N_D^\sigma(v)=\{u,w\}$. 
Since $D$ is neighborhood-cyclic, 
the vertices $u$ and $w$ are joined by a double arc, which is a contradiction. 
It follows that $d_D(v)\geq 6$.
Moreover, we have that $d_D^+(v)\geq 3$ and $d_D^-(v)\geq 3$, and 
if $u$ sends charge to a neighbor $w$, 
then $N_D^\sigma(w)=\{u,v\}$ for some $\sigma\in \{ +,-\}$.
By the definition of Phase 3 and Phase 4, 
the vertex $u$ sends at most ${\rm demand}(w,\sigma)\cdot \frac{d_D(u)-5}{d_D(u)+d_D(v)-10}$ charge to $w$.
If there are at most two pairs $(w,\sigma)$
served by $u$, 
then $c(u)\geq 6-2\cdot \frac{2}{3}\cdot \frac{1}{2}=\frac{16}{3}$.
If $u$ sends charge to exactly three neighbors $w_1$, $w_2$, and $w_3$, then we claim that $d_D(v)\geq 7$.
Suppose, for a contradiction, that $d_D(v)=6$.
Since $d_D^+(v)=d_D^-(v)=3$, there is $\sigma$ in $\{+,-\}$ such that, by symmetry, 
$N^\sigma_D(v)=\{u,w_1,w_2\}$
and $w_1,w_2\in N^\sigma_D(u)$. 
Since $N_D^\sigma(v)$ is cyclic, the vertices $w_1$ and $w_2$ are joined by a double arc, 
which is a contradiction to $d_D^{-\sigma}(w_1)=2$. 
Hence, we have $d_D(v)\geq 7$ and $c(u)\geq 6-3\cdot \frac{2}{3}\cdot \frac{6-5}{6+7-10}=\frac{16}{3}$.
If $u$ sends charge to exactly four neighbors $w_1$, $w_2$, $w_3$, and $w_4$, then we claim that $d_D(v)\geq 8$. 
Suppose, for a contradiction, that $d_D(v)\leq 7$.
Let $\sigma$ in $\{+,-\}$ be such that $d_D^\sigma(v)=3$.
Since $D$ is neighborhood-cyclic and strongly $2$-connected, 
it follows that $d_D^+(u)=d_D^-(u)=3$.
By symmetry, we may assume that $N_D^\sigma(u)=\{v,w_1,w_2\}$.
It holds that $N_D^\sigma(v)=\{ u,w_1,w_2\}$.
Since $N_D^\sigma(v)$ is cyclic, the vertices $w_1$ and $w_2$ are joined by a double arc, 
which is a contradiction to $d_D^{-\sigma}(w_1)=2$.
Hence, we have $d_D(v)\geq 8$ and $c(u)\geq 6-4\cdot \frac{2}{3}\cdot \frac{6-5}{6+8-10}=\frac{16}{3}$.

\medskip

\noindent {\bf Case 2.} {\it $u$ is incident with exactly two double arcs.}

\medskip 

\noindent Let $D$ contain the two double arcs between $u$ and $v_1$ and between $u$ and $v_2$.
By symmetry, we assume that $d_D(v_1)\leq d_D(v_2)$. 
Furthermore, let $w_1$ and $w_2$ be the other two neighbors of $u$.

First, let $d_D(v_1)=4$.
In this case $v_1$ and $v_2$ are joined by a double arc. 
Since $D-v_2$ is strongly connected, we have $d_D^+(u)=d_D^-(u)=3$.
Since $D-u$ is strongly connected, we have $d_D(v_2)\geq 6$.
If $u$ sends charge to neither $w_1$ nor $w_2$, then $c(u)\geq 6-2\cdot \frac{2}{3}\cdot \frac{6-5}{6+6-10}=\frac{16}{3}$.
If $u$ sends charge to both $w_1$ and $w_2$, then $d_D(w_1)\leq 5$ and $d_D(w_2)\leq 5$. 
Let $\sigma$ in $\{+,-\}$ be such that $N_D^\sigma(w_1)=\{u,v_2\}$ and $N_D^{-\sigma}(w_2)=\{u,v_2\}$. 
Since $D-w_1$ and $D-w_2$ are strongly connected, it follows that $d_D(v_2)\geq 8$. 
Thus, we have $c(u)\geq 6-4\cdot \frac{2}{3}\cdot\frac{6-5}{6+8-10}=\frac{16}{3}$.
If $u$ sends charge to $w_1$ but not to $w_2$, then $d_D(w_1)\leq 5$. 
Since $N_D^\sigma(w_1)=\{u,v_2\}$ for some $\sigma\in \{ +,-\}$, and $D-w_1$, $D-w_2$, and $D-v_2$ are strongly connected, 
it follows that $d_D(v_2)\geq 7$. 
In this case, we have $c(u)\geq 6-3\cdot\frac{2}{3}\cdot \frac{6-5}{6+7-10}=\frac{16}{3}$.
Hence, we may assume $d_D(v_1)\geq 5$.

Now, let $d_D(v_1)=5$ and $d_D(v_2)=5$.
Since $d_D(v_1)=5$, the vertices $v_1$ and $v_2$ are joined by a single or double arc. 
First, suppose that $v_1$ and $v_2$ are joined by a single arc.
It follows that $d_D^+(u)=d_D^-(u)=3$. 
By symmetry, we may assume $(v_1,v_2)\in A(D)$.
Let $N_D^-(u)=\{v_1,v_2,w_1\}$ and $N_D^+(u)=\{v_1,v_2,w_2\}$. 
Since $N_D^-(u)$, $N_D^-(v_1)$, and $N_D^+(v_2)$ are cyclic, it follows that $v_1v_2w_1v_1$ is a directed cycle in $D$. 
Since $N_D^+(u)$, $N_D^-(v_1)$, and $N_D^+(v_2)$ are cyclic, it follows that $v_1v_2w_2v_1$ is a directed cycle in $D$. 
Since $D-w_2$ is strongly connected and $N_D^-(w_1)$ is cyclic, it follows that $d_D(w_1)\geq 6$. 
Since $D-w_1$ is strongly connected and $N_D^+(w_2)$ is cyclic, it follows that $d_D(w_2)\geq 6$. 
Hence, if $v_1$ and $v_2$ are joined by a single arc, then we have $c(u)\geq 6-\frac{1}{3}-\frac{1}{3}=\frac{16}{3}$.
Assume now that $v_1$ and $v_2$ are joined by a double arc.
Since $D$ is strongly 2-connected, 
there is no vertex distinct from $v_1$ and $v_2$
to which $u$ sends charge. 
Hence, we have $c(u)\geq 6-{\rm demand}(v_1,\sigma)-{\rm demand}(v_2,-\sigma)=6-\frac{1}{3}-\frac{1}{3}=\frac{16}{3}$.

Now, let $d_D(v_1)=5$ and $d_D(v_2)\geq 6$. 

First, suppose that $v_1$ and $v_2$ are not joined by a double arc.
It follows that $d_D^+(u)=3$ and $d_D^-(u)=3$, and that $v_1$ and $v_2$ are joined by a single arc. 
By symmetry, we may assume that $d_D^+(v_1)=2$ and, hence, $(v_1,v_2)\in A(D)$.
Let $N_D^-(u)=\{w_1,v_1,v_2\}$ and $N^+_D(u)=\{w_2,v_1,v_2\}$.
Suppose that $d_D(w_1)=4$. 
Note that $w_1$ and $v_1$ are not joined by a double arc, since $d_D^+(v_1)=2$. 
If $w_1$ and $v_2$ are joined by a double arc, then $d_D(v_2)\geq 7$ and $c(u)\geq 6-\frac{2}{3}\cdot 2\cdot \frac{6-5}{6+7-10}>\frac{16}{3}$. Hence, we may assume that $w_1$ and $v_2$ are not joined by a double arc.
Since $N_D^-(u)$ is cyclic, it follows therefore that $w_1v_1v_2w_1$ is a directed cycle in $D$. 
Let $N_D^-(w_1)=\{v_2,x\}$. 
Note that $x=w_2$ is possible. 
Let $N_D^-(v_1)=\{u,w_1,y\}$. 
Note that $y\neq v_2$.
Since $N_D^-(v_1)$ is cyclic, it follows that $w_1uyw_1$ is a directed cycle in $D$. 
This implies that $x=y=w_2$.
Since $N_D^-(w_1)$ is cyclic, the vertices $v_2$ and $w_2$ are joined by a double arc. 
Furthermore, since $D$ is strongly $2$-connected, we obtain that $d_D(v_2)\geq 7$ and $d_D(w_2)\geq 6$. 
It follows that $c(u)\geq 6-{\rm demand}(w_1,-)=\frac{16}{3}$.
Hence, we may now assume that $d_D(w_1)\geq 5$.
Suppose that $d_D(w_2)=4$.
Since $d_D^+(v_1)=2$, we have that $N_D^-(w_2)=\{u,v_2\}$.
If $(w_2,v_1)\in A(D)$ and $x$ is such that $N_D^+(w_2)=\{v_1,x\}$, 
then $v_1$ and $x$ are joined by a double arc, which is a contradiction to $d_D^+(v_1)=2$ or $v_1$ and $v_2$ are joined by a single arc. 
Hence, we have that $(w_2,v_1)\notin A(D)$.
Since $N_D^+(u)$ is cyclic, it follows that $w_2$ and $v_2$ are joined by a double arc.
Let $N_D^+(w_2)=\{v_2,x\}$. 
Note that $x\neq v_1$.
This implies that $v_2$ and $x$ are joined by a double arc.
It follows that $d_D(v_2)\geq 7$,
which implies $c(u)\geq 6-{\rm demand}(w_1,+)-{\rm demand}(w_2,+)\cdot \frac{6-5}{6+7-10}\geq6-\frac{1}{3}-\frac{2}{3}\cdot \frac{1}{3}\geq \frac{16}{3}$. 
Hence, we may now assume that $d_D(w_1)\geq 5$ and $d_D(w_2)\geq 5$.
If $u$ sends charge to at most one of $w_1$ and $w_2$, 
then we have $c(u)\geq 6-\frac{1}{3}-{\rm demand}(v_1,+)\geq \frac{16}{3}$.
Thus, we may assume that $u$ sends charge to both $w_1$ and $w_2$,
which implies $d_D(w_1)=5$ and $d_D(w_2)=5$.
Again, it follows that $N_D^-(w_2)=\{u,v_2\}$, since $d_D^+(v_1)=2$. 
Hence, we obtain $c(u)\geq 6-{\rm demand}(w_1,+)-\frac{1}{2}{\rm demand}(v_1,+)-\frac{1}{2}{\rm demand}(w_2,-)=\frac{16}{3}$.

Next, suppose that $v_1$ and $v_2$ are joined by a double arc.
First, suppose that $d_D(v_2)=6$.
Since $D$ is strongly $2$-connected, there is at most one vertex $w$ other than $v_1$ to which $u$ sends charge.
If such a vertex $w$ does not exist or $d_D(w)=5$, then $c(u)\geq \frac{16}{3}$.
Hence, we may assume that such a vertex $w$ exists and that $d_D(w)=4$.
In view of Phase 2, it follows that $c(u)\geq 6-\frac{2}{3}=\frac{16}{3}$.
Hence, we may assume that $d_D(v_2)\geq 7$.
Note that $u$ sends no charge to $v_1$, since $d_D(v_2)>d_D(u)$.
There are at most two vertices to which $u$ sends charge. 
In view of Phase 1, it now follows easily that $c(u)\geq \frac{16}{3}$.

Finally, let $d_D(v_1)\geq 6$ and $d_D(v_2)\geq 6$. 
Note that $u$ sends at most ${\rm demand}(w,\sigma)\cdot \frac{6-5}{6+6-10}$ charge to any neighbor $w$,
which implies $c(u)\geq 6-2\cdot \frac{2}{3}\cdot \frac{6-5}{6+6-10}\geq \frac{16}{3}$.

\medskip

\noindent {\bf Case 3.} {\it $u$ is incident with three double arcs.}

\medskip 

\noindent Let $N_D^+(u)=N_D^-(u)=\{ v_1,v_2,v_3\}$.

Assume first that $v_1$ and $v_2$ are joined by a double arc. 

If $d_D(v_1)=4$, then $d_D(v_2)\geq 6$. 
If $u$ sends no charge to $v_3$, we have $c(u)\geq 6-2\cdot \frac{2}{3}\cdot \frac{6-5}{6+6-10}=\frac{2}{3}$.
Hence, suppose, for a contradiction, that $u$ sends charge to $v_3$.
Let $\sigma$ in $\{+,-\}$ be such that $d_D^\sigma(v_3)=2$.
We have $N_D^\sigma(v_3)=\{u,v_2\}$.
But now $D-v_2$ is not strongly connected, which is a contradiction.

If $d_D(v_1)=5$ and $d_D(v_2)=5$, then we claim that $d_D(v_3)\geq 6$. 
Suppose not, that is, $d_D(v_3)\leq 5$. 
Let $\sigma$ in $\{+,-\}$ be such that $d_D^\sigma(v_3)=2$. 
By symmetry, we may assume that $N_D^\sigma(v_3)=\{u,v_1\}$. 
But now, since $d_D(v_2)=5$, it follows that $D-v_3$ is not strongly connected, which is a contradiction. 
Hence, we have $d_D(v_3)\geq 6$ and $c(u)\geq 6-2\cdot\frac{1}{3}=\frac{16}{3}$.

If $d_D(v_1)=5$ and $d_D(v_2)\geq 6$, then we claim that $d_D(v_3)\geq 5$. 
Indeed, if $d_D(v_3)=4$, then $N_D^\sigma(v_3)=\{u,v_2\}$ and $N_D^{-\sigma}(v_3)=\{u,v_i\}$ 
for some $\sigma\in \{+,-\}$ and $i\in \{1,2\}$. 
In all cases, the graph $D-v_2$ is not strongly connected, which is a contradiction. 
Thus, we have $d_D(v_3)\geq 5$ and $c(u)\geq 6-\frac{1}{6}-\frac{1}{3}\geq \frac{16}{3}$.

If $d_D(v_1)\geq 6$ and $d_D(v_2)\geq 6$, then $d_D(v_3)\geq 5$ or $d_D(v_3)=4$,
and $N_D^\sigma(v_3)=\{u,v_i\}$ in the latter case.
We have $c(u)\geq \min \left\{6-\frac{1}{3}, 6-2\cdot\frac{2}{3}\cdot \frac{6-5}{6+6-10}\right\}\geq \frac{16}{3}$.

Now assume 
that there is no double arc in $\{ v_1,v_2,v_3\}$ and 
that $v_1v_2v_3v_1$ is a directed cycle in $D$. 
If $d_D(v_1)=4$, then, since $D$ is strongly $2$-connected, we have that $d_D(v_2)\geq 6$ and $d_D(v_3)\geq 6$,
and, hence, $c(u)\geq 6-2\cdot \frac{2}{3}\cdot \frac{6-5}{6+6-10}\geq \frac{16}{3}$. 
Hence, by symmetry, we may assume that $d_D(v_1),d_D(v_2),d_D(v_3)\geq 5$.
Since $D$ is strongly $2$-connected, it is not possible that $d_D(v_1)=d_D(v_2)=d_D(v_3)=5$. 
Therefore, we obtain $c(u)\geq 6-\frac{1}{3}-\frac{1}{3}\geq \frac{16}{3}$.

This completes the proof of the claim.
\end{proof}

\begin{claim}\label{claim5}
If $u$ is a vertex in $V_7$, 
then $c(u)\geq \frac{16}{3}$. 
\end{claim}
\begin{proof}[Proof of Claim \ref{claim5}]
Since 
$7-2\cdot \frac{2}{3}\geq \frac{16}{3}$,
we may assume that 
$u$ sends charge along at least three arcs.
Again, there is at least one neighbor $v$ of $u$ such that $u$ and $v$ are joined by a double arc.

\medskip

\noindent {\bf Case 1.} {\it $u$ is incident with exactly one double arc.}

\medskip

\noindent Let $D$ contain a double arc between $u$ and $v$.
First observe that $u$ sends no charge in Phase 1 and Phase 2.
Next, suppose, for a contradiction, that $d_D(v)\leq 5$.
Let $\sigma$ in $\{+,-\}$ be such that $d_D^\sigma(v)=2$.
Let $N_D^\sigma(v)=\{u,w\}$. 
Since $D$ is neighborhood-cyclic, 
the vertices $u$ and $w$ are joined by a double arc, which is a contradiction. 
It follows that $d_D(v)\geq 6$.
Moreover, we have that $d_D^+(v)\geq 3$ and $d_D^-(v)\geq 3$, and 
if $u$ sends charge to a neighbor $w$, 
then $N_D^\sigma(w)=\{u,v\}$ for some $\sigma\in \{ +,-\}$.
By the definition of Phases 3 and 4,
the vertex $u$ sends at most ${\rm demand}(w,\sigma)\cdot \frac{d_D(u)-5}{d_D(u)+d_D(v)-10}$ charge to $w$.
If $u$ sends charge to at least three neighbors 
$w_1$, $w_2$, and $w_3$, 
then we claim that $d_D(v)\geq 7$.
Suppose, for a contradiction, that $d_D(v)\leq 6$.
By symmetry, let $\sigma\in\{+,-\}$ such that $w_1,w_2\in N_D^\sigma(u)$. 
This implies that $N_D^\sigma(v)=\{u,w_1,w_2\}$ as $d_D^\sigma(v)=3$.
Since $N_D^\sigma(v)$ is cyclic, 
it follows that $w_1$ and $w_2$ are joined by a double arc. 
This implies that $d_D^{-\sigma}(w_1)\geq 3$ and $d_D^{-\sigma}(w_2)\geq 3$ which contradicts
that $u$ sends charge to $w_1$ and $w_2$.
Hence, we have $d_D(v)\geq 7$ and $c(u)\geq 7-5\cdot \frac{2}{3}\cdot \frac{7-5}{7+7-10}\geq \frac{16}{3}$.

\medskip

\noindent {\bf Case 2.} {\it $u$ is incident to exactly two double arcs.}

\medskip

\noindent Let $D$ contain the two double arcs between $u$ and $v_1$ and between $u$ and $v_2$. By symmetry, we assume that $d_D(v_1)\leq d_D(v_2)$. Furthermore, let $w_1,w_2$, and $w_3$ be the other three neighbors of $u$. 

First, assume that $d_D(v_1)=4$. 
In this case, the vertices $v_1$ and $v_2$ are joined by a double arc. 
Since $D-u$ is strongly connected, we have $d_D(v_2)\geq 6$. 
If $u$ sends charge to $w_1$ but not to $w_2$ and $w_3$, 
then $c(u)\geq 7-2\cdot\frac{2}{3}\cdot \frac{7-5}{7+6-10}-\frac{2}{3}\cdot \frac{7-5}{7+6-10}\geq\frac{16}{3}$.
If $u$ sends charge to $w_1$ and $w_2$, 
we claim that $d_D(v_2)\geq 7$.
Suppose, for a contradiction, that $d_D(v_2)\leq 6$.
Since $D-u$ is strongly connected,
we obtain that $d_D^+(v_2)=d_D^-(v_2)=3$,
$N_D^\sigma(w_1)=\{u,v_2\}$, and 
$N_D^{-\sigma}(w_2)=\{u,v_2\}$
for some $\sigma$ in $\{ +,-\}$.
Since $d_D(u)=7$, it follows that $D-w_1$ or $D-w_2$ is not strongly connected,
which is a contradiction.
Hence, we have $d_D(v_2)\geq 7$ and $c(u)\geq 7-5\cdot \frac{2}{3}\cdot \frac{7-5}{7+7-10}\geq \frac{16}{3}$.

Now, assume that $d_D(v_1)=5$ and $d_D(v_2)=5$.
If $v_1$ and $v_2$ are joined by a double arc, 
the vertex $u$ sends charge to at most one vertex of $w_1,w_2,$ and $w_3$ 
as $D$ is strongly $2$-connected. 
In this case, we have $c(u)\geq 7-2\cdot \frac{1}{3}-\frac{2}{3}\geq \frac{16}{3}$.
Hence, we may assume that $v_1$ and $v_2$ are not joined by a double arc. 
This implies that $d_D^+(u),d_D^-(u)\geq 3$.
Since $d_D(v_1)=5$, the vertices $v_1$ and $v_2$ are joined by a single arc.
By symmetry, we may assume that $(v_1,v_2)\in A(D)$.
It follows that $d^+_D(v_1)=2$ and $d_D^-(v_2)=2$.
If $u$ sends charge to neither $w_2$ nor $w_3$, 
then $c(u)\geq 7-2\cdot \frac{1}{3}-\frac{2}{3}\geq \frac{16}{3}$.
Suppose that $u$ sends charge to $w_1$ and $w_2$.
If $w_1,w_2\in N_D^-(u)$, then $N_D^+(w_1)=N_D^+(w_2)=\{u,v_1\}$. 
Since $N_D^+(u)=\{v_1,v_2,w_3\}$ is cyclic and $d_D^-(v_2)=2$, 
it holds that $v_1v_2w_3v_1$ is a directed cycle in $D$,
which implies $d_D(v_1)\geq 6$, a contradiction to $d_D(v_1)=5$.
Similarly, it follows that $w_1$ and $w_2$ are not both in $N_D^+(u)$.
This already implies that $u$ sends no charge to $w_3$.
Hence, we assume that $w_1\in N_D^-(u)$ and $w_2\in N_D^+(u)$. 
By symmetry, we may assume that $d_D^-(u)=3$. 
Since $N_D^-(u)$ is cyclic and $d_D^+(v_1)=2$, 
it follows that $v_1v_2w_1v_1$ is a directed cycle in $D$. 
Suppose, for a contradiction, that $d_D(w_1)=4$. 
Since $N_D^-(w_1)=\{v_2,x\}$ is cyclic, we obtain $d_D(v_2)\geq 6$, which is a contradiction to $d_D(v_2)=5$.
Hence, we have $d_D(w_1)\geq 5$ and 
$c(u)\geq 7-3\cdot\frac{1}{3}-\frac{2}{3}=\frac{16}{3}$.

Now, assume that $d_D(v_1)=5$ and $d_D(v_2)=6$. 

\noindent First, assume that $v_1$ and $v_2$ are joined by a double arc.
If $u$ sends charge in Phase 1, 
then there are vertices $x$ and $y$ such that $N_D^\sigma(x)=\{v_1,v_2\}$ and $N_D^{-\sigma}(y)=\{u,v_2\}$. 
Note that $u$ sends ${\rm demand}(y,-\sigma)$ charge to $y$ and ${\rm demand}(v_1,-\sigma)$ to $v_1$.
Moreover, the vertex $u$ sends no charge to any other neighbor.
We have $c(u)\geq 7-\frac{2}{3}-\frac{1}{3}\geq \frac{16}{3}$.
If $u$ sends no charge in Phase 1, 
then $u$ sends charge to $v_1$ and at most two other neighbors in Phases 3 and 4 as $D$ is strongly $2$-connected.
This implies that $c(u)\geq 7-\frac{1}{3}-2\cdot \frac{2}{3}\geq\frac{16}{3}$.

\noindent Now, assume that $v_1$ and $v_2$ are not joined by a  double arc. 
Since $d_D(v_1)=5$, the vertices $v_1$ and $v_2$ are joined by a single arc.
By symmetry, we may assume that $(v_1,v_2)\in A(D)$.
Since $N_D^+(u)$ and $N_D^-(u)$ are cyclic, 
we have $d_D^+(u)\geq 3$ and $d_D^-(u)\geq 3$. 
Assume first $d_D^-(u)=3$.
Let $N_D^-(u)=\{v_1,v_2,w_1\}$ and $N_D^+(u)=\{v_1,v_2,w_2,w_3\}$.

It follows that $w_1$ and $v_2$ are joined by a double arc or $v_1v_2w_1v_1$ is a directed cycle in $D$. 
First, we assume that $w_1$ and $v_2$ are joined by a double arc.
Let $i\in\{2,3\}$.
If $u$ sends charge to $w_i$, 
then $N_D^-(w_i)=\{u,v_2\}$ since $d_D^+(v_1)=2$.
As $d_D(v_2)=6$, it follows that $u$ sends charge to at most one of $w_2$ and $w_3$,
and we have $c(u)\geq 7-\frac{1}{3}-2\cdot \frac{2}{3}\geq \frac{16}{3}$.
Now, we assume that 
there is no double arc
between $w_1$ and $v_2$
but that 
$v_1v_2w_1v_1$ is a directed cycle in $D$.
Suppose, for a contradiction, that $d_D(w_1)=4$.
Let $N_D^-(w_1)=\{v_2,x\}$. 
The vertices $v_2$ and $x$ are joined by a double arc. 
Let $N_D^-(v_1)=\{u,w_1,y\}$.
Note that $y\neq v_2$.
Since $N_D^-(v_1)$ is cyclic and $d_D(w_1)=4$, we have $x=y=w_2$.
It follows that $D-w_2$ is not strongly connected, 
which is a contradiction.
Hence, we have $d_D(w_1)\geq 5$.
Consider the case $d_D(w_2)=d_D(w_3)=4$.
Since $d_D^+(v_1)=2$, we have $N_D^-(w_2)=\{v_2,u\}$ and $N_D^-(w_3)=\{v_2,u\}$.
Hence, we have $c(u)\geq 7-\frac{1}{3}-\frac{1}{3}-2\cdot \frac{2}{3}\cdot \frac{7-5}{7+6-10}\geq \frac{16}{3}$.
Finally, if $d_D(w_1)\geq 5$ and $d_D(w_2)\geq 5$, 
then we have $c(u)\geq 7-3\cdot\frac{1}{3}-\frac{2}{3}\geq \frac{16}{3}$. 
Now, assume that $d_D^-(u)=4$. Let $N_D^-(u)=\{w_1,w_2,v_1,v_2\}$ and $N_D^+(u)=\{v_1,v_2,w_3\}$.
It follows that $v_2$ and $w_3$ are joined by a double arc or $v_1v_2w_3v_1$ is a directed cycle in $D$. 
First, assume that $v_2$ and $w_3$ are joined by a double arc.
If $u$ sends no charge to $w_3$, then $c(u)\geq 7-\frac{1}{3}-2\cdot \frac{2}{3}\geq \frac{16}{3}$.
Now, assume that $u$ sends charge to $w_3$. 
Note that $d_D^-(w_3)=2$.
Suppose, for a contradiction, that $u$ sends charge to $w_1$ and $w_2$ and  $N_D^+(w_1)=N_D(w_2)=\{u,v_1\}$.
The cyclicity of $N_D^-(v_1)$ implies that $w_1$ and $w_2$ are joined by a double arc, which contradicts $d_D^+(w_1)=2$ and $d_D^+(w_2)=2$.
Next, suppose, for a contradiction, that $u$ sends charge to $w_1$ and $w_2$ and $N_D^+(w_1)=\{u,v_2\}$.
Since $d_D(v_2)=6$, it follows that $N_D^+(w_2)=\{u,v_1\}$.
Moreover, since $d_D^+(v_1)=2$ and $d_D(v_2)=6$, the cyclicity of $N_D^-(u)$ implies that $d_D^+(w_1)\geq 3$ or $d_D^+(w_2)\geq 3$, which is a contradiction to $d_D^+(w_1)=2$ and $d_D^+(w_2)=2$.
Hence, the vertex $u$ sends charge to at most one of $w_1$ and $w_2$,
and we have $c(u)\geq 7-\frac{1}{3}-2\cdot \frac{2}{3}\geq \frac{16}{3}$.
Now, we assume that 
there is no double arc 
between $v_2$ and $w_3$
but that 
$v_1v_2w_3v_1$ is a directed cycle in $D$.
Suppose, for a contradiction, that $d_D(w_3)=4$. 
Let $N_D^+(w_3)=\{v_1,x\}$. 
The vertices $v_1$ and $x$ are joined by a double arc. 
This implies that $d_D(v_1)\geq 6$, which is a contradiction to $d_D(v_1)=5$.
Hence, we have $d_D(w_3)\geq 5$.
Suppose that $d_D(w_1)=d_D(w_2)=4$. 
Since $v_1v_2w_3v_1$ is a directed cycle in $D$ and $d_D(v_1)=5$, it holds that $N_D^+(w_1)=\{u,v_2\}$ or $N_D^+(w_2)=\{u,v_2\}$. 
If $N_D^+(w_1)=N_D^+(w_2)=\{u,v_2\}$, then $c(u)\geq 7-2\cdot\frac{2}{3}\cdot \frac{7-5}{7+6-10}-\frac{1}{3}-\frac{1}{3}\geq \frac{16}{3}$.
By symmetry, we may assume that $N_D^+(w_1)=\{u,v_1\}$ and $N_D^+(w_2)=\{u,v_2\}$. 
Since $N_D^-(v_1)=\{w_1,u,w_3\}$ is cyclic, it follows that $(w_3,w_1)\in A(D)$. 
Let $N_D^-(w_1)=\{w_3,x\}$.
If $x=v_2$, then $D-w_3$ is not strongly connected, which is a contradiction.
If $x\neq v_2$, then $x$ and $w_3$ are joined by a double arc. 
This implies that $d_D(w_3)\geq 6$. 
Hence, $c(u)\geq 7-\frac{2}{3}-\frac{2}{3}\cdot \frac{7-5}{7+6-10}-\frac{1}{3}\geq \frac{16}{3}$.
If $d_D(w_1)\geq 5$ or $d_D(w_2)\geq 5$, we have $c(u)\geq 7-\frac{2}{3}-3\cdot\frac{1}{3}\geq \frac{16}{3}$.

Next, we assume that $d_D(v_1)=5$ and $d_D(v_2)\geq 7$.
Consider first that $v_1$ and $v_2$ are joined by a double arc. 
There is at most one neighbor $w$ of $u$ such that $N_D^\sigma(w)=\{u,v_1\}$ and hence, 
we have $c(u)\geq 7-\frac{2}{3}-\frac{1}{6}-2\cdot \frac{2}{3}\cdot\frac{7-5}{7+7-10}\geq \frac{16}{3}$.
Consider now the case that $v_1$ and $v_2$ are only joined by a single arc.
By symmetry, we may assume $(v_1,v_2)\in A(D)$.
Note that this implies that $d_D^+(u)\geq 3$ and $d_D^-(u)\geq 3$.
If $u$ sends charge to at most two of $w_1,w_2$, and $w_3$, then 
$c(u)\geq 7-\frac{1}{6}-2\cdot\frac{2}{3}\geq \frac{16}{3}$.
Thus, we may assume that $u$ sends charge to $w_1,w_2$, and $w_3$.
If there is at most one neighbor $w$ such that $N_D^+(w)=\{u,v_1\}$, 
then $c(u)\geq 7-\frac{2}{3}-\frac{1}{6}-2\cdot \frac{2}{3}\cdot \frac{7-5}{7+7-10}\geq \frac{16}{3}$.
Suppose there are two neighbors $w_1$ and $w_2$ such that $N_D^+(w_1)=N_D^+(w_2)=\{u,v_1\}$.
This implies that $d_D^+(u)=3$.
Let $N_D^+(u)=\{v_1,v_2,w_3\}$.
Since $d_D(v_1)=5$ and $N_D^+(u)$ is cyclic, we obtain that $v_2$ and $w_3$ are joined by a double arc.
If $d_D(w_3)=4$, then, for $N_D^+(w_3)=\{v_2,x\}$, 
the vertices $v_2$ and $x$ are joined by a double arc. 
Since $D-x$ is strongly connected, it follows that $d_D(v_2)\geq 8$.
This implies that $u$ sends no charge to $v_1$.
Hence, we have $c(u)\geq 7-2\cdot\frac{2}{3}-\frac{2}{3}\cdot \frac{7-5}{7+8-10}\geq \frac{16}{3}$.
If $d_D(w_3)\geq 5$, 
then we have $c(u)\geq 7-2\cdot \frac{2}{3}-2\cdot\frac{1}{6}\geq \frac{16}{3}$.

If $d_D(v_1)\geq 6$ and $d_D(v_2)\geq 6$, then $u$ sends at most ${\rm demand}(w,\sigma)\cdot\frac{7-5}{7+6-10}$ charge to a neighbor $w$.
Hence, $c(u)\geq 7-3\cdot \frac{2}{3}\cdot \frac{7-5}{7+6-10}\geq \frac{16}{3}$.

\medskip

\noindent {\bf Case 3.} {\it $u$ is incident with exactly three double arcs.}

\medskip

\noindent Let $D$ contain the three double arcs between $u$ and $v_1$, between $u$ and $v_2$, and between $u$ and $v_3$. 
Let $w$ be the fourth neighbor of $u$.
First, we assume that there is a directed cycle of length 3 in $D[\{v_1,v_2,v_3\}]$. 
By symmetry, we may assume that $v_1v_2v_3v_1$ is a directed cycle in $D$. 
Suppose that $d_D(v_1)=4$.
Consider the case that $u$ sends no charge to $w$.
Since $D$ is strongly 2-connected, we have $d_D(v_2)\geq 5$ and $d_D(v_3)\geq 6$ or $d_D(v_2)\geq 6$ and $d_D(v_3)\geq 5$. 
Hence, we have $c(u)\geq 7-\frac{2}{3}\cdot \frac{7-5}{7+6-10}-\frac{2}{3}-\frac{1}{3}\geq \frac{16}{3}$.
Consider now the case that $u$ sends charge to $w$.
Since $D$ is strongly $2$-connected, we have $d_D(v_2)\geq 6$ and $d_D(v_3)\geq 6$ or 
$d_D(v_2)\geq 5$ and $d_D(v_3)\geq 7$ or 
$d_D(v_2)\geq 7$ and $d_D(v_3)\geq 5$.
In all cases, we have 
$c(u)\geq \min\left\{7-3\cdot \frac{2}{3}\cdot\frac{7-5}{7+6-10}, 7-\frac{2}{3}-2\cdot \frac{2}{3}\cdot\frac{7-5}{7+7-10}-\frac{1}{3}\right\} \geq \frac{16}{3}$.
If $d_D(v_1), d_D(v_2),d_D(v_3)\geq 5$, 
then $c(u)\geq 7-3\cdot \frac{1}{3}-\frac{2}{3}\geq \frac{16}{3}$.

Now, we assume that $v_1$ and $v_2$ are joined by a double arc. 
Suppose that $d_D(v_1)=4$. 
Since $D$ is strongly $2$-connected, 
it follows that $d_D(v_2)\geq 6$. 
If $d_D(v_3)\geq 6$, then $c(u)\geq 7-3\cdot \frac{2}{3}\cdot\frac{7-5}{7+6-10}\geq \frac{16}{3}$.
Suppose, for a contradiction, that $d_D(v_3)=4$. 
Then $v_2$ and $v_3$ are joined by a double arc.
But then $D-v_2$ is not strongly connected, which is a contradiction.
If $d_D(v_3)=5$, then $N_D^\sigma(v_3)=\{u,v_2\}$ for some $\sigma\in\{+,-\}$. 
Since $D-v_2$ is strongly connected, 
we have $w\in N_D^\sigma(u)$.
Moreover, if $u$ sends charge to $w$, then $N_D^{-\sigma}(w)=\{u,v_2\}$ and $d_D(v_2)\geq 7$.
We have $c(u)\geq \min\left\{7-2\cdot\frac{2}{3}\cdot\frac{7-5}{7+6-10}-\frac{1}{3},7-3\cdot \frac{2}{3}\cdot\frac{7-5}{7+7-10}-\frac{1}{6}\right\}\geq \frac{16}{3}$.
Let $d_D(v_1)=5$ and $d_D(v_2)=5$.
Suppose, for a contradiction, that $d_D(v_3)=4$. 
We obtain $N_D^\sigma(v_3)=\{v_1,u\}$ and $N_D^{-\sigma}(v_3)=\{u,v_2\}$.
Now, 
$D-u$ is not strongly connected, which is a contradiction.
Hence, we have $d_D(v_3)\geq 5$ and $c(u)\geq 7-3\cdot \frac{1}{3}-\frac{2}{3}\geq\frac{16}{3}$.
Let $d_D(v_1)=5$ and $d_D(v_2)\geq 6$.
If $d_D(v_3)=4$, 
then $v_2$ and $v_3$ are joined by a double arc as $D-v_2$ is strongly connected.
Since $D$ is strongly $2$-connected, we obtain that $d_D(v_2)\geq 8$.
We have $c(u)\geq 7-\frac{2}{3}-2\cdot\frac{2}{3}\cdot\frac{2}{5}\geq \frac{16}{3}$.
If $d_D(v_3)\geq 5$, 
then $c(u)\geq 7-\frac{1}{3}-\frac{1}{3}-\frac{2}{3}\geq \frac{16}{3}$.
If $d_D(v_1)\geq 6$, then $c(u)\geq 7-\frac{2}{3}\geq \frac{16}{3}$.

This completes the proof of the claim.
\end{proof}

\begin{claim}\label{claim6}
If $u$ is a vertex in $V_8$, then $c(u)\geq \frac{16}{3}$.
\end{claim}
\begin{proof}[Proof of Claim \ref{claim6}]
Since $8-4\cdot\frac{2}{3}\geq \frac{16}{3}$,
we may assume that $u$ sends charge along at least five arcs.
Again, there is at least one neighbor $v$ of $u$ 
such that $u$ and $v$ are joined by a double arc.

\medskip

\noindent {\bf Case 1.} {\it $u$ is incident with exactly one double arc.}

\medskip
    
\noindent Let $D$ contain a double arc between $u$ and $v$. 
Observe that $u$ sends no charge in Phase 1 or Phase 2.
If $u$ sends charge along the arc $(u,w)$, then $N_D^-(w)=\{u,v\}$.
If $u$ sends charge along the arc $(w,u)$, then $N_D^+(w)=\{u,v\}$.
Hence, we obtain $d_D(v)\geq 7$ and 
$c(u)\geq 8-6\cdot\frac{2}{3}\cdot\frac{8-5}{8+7-10}\geq \frac{16}{3}$.

\medskip

\noindent {\bf Case 2.} {\it $u$ is incident with exactly two double arcs.}

\medskip

\noindent Let $D$ contain the two double arcs between $u$ and $v_1$ and between $u$ and $v_2$.
By symmetry, we assume that $d_D(v_1)\leq d_D(v_2)$.
Note that $d_D(v_1)=4$ implies that $v_1$ and $v_2$ are joined by a double arc.
Since $u$ sends charge along at least five arcs,
we obtain that $d_D(v_2)\geq 7$.
Hence, we have $c(u)\geq 8-6\cdot\frac{2}{3}\cdot \frac{8-5}{8+7-10}\geq \frac{16}{3}$.
Now, let $d_D(v_1)=5$ and $d_D(v_2)\geq 5$.
Since $d_D(v_1)=5$, the vertices $v_1$ and $v_2$ are joined by a single arc or a double arc.
If $v_1$ and $v_2$ are joined by a double arc, then there is at most one neighbor $w\neq v_2$ of $u$ such that $N_D^\sigma(w)=\{u,v_1\}$. 
Since $u$ sends charge along at least five arcs, this implies that there are at least two vertices $w_1$ and $w_2$ such that $N_D^{\sigma_1}(w_1)=N_D^{\sigma_2}(w_2)=\{u,v_2\}$. 
Hence, we obtain that $d_D(v_2)\geq 6$ and $c(u)\geq 8-\frac{2}{3}-\frac{1}{3}-3\cdot \frac{2}{3}\cdot \frac{8-5}{8+6-10}\geq \frac{16}{3}$.
Assume now that $v_1$ and $v_2$ are joined by a single arc.
By symmetry, we may assume that $(v_1,v_2)\in A(D)$.
Suppose, for a contradiction, that there are two vertices $w_1$ and $w_2$ such that $N_D^+(w_1)=N_D^+(w_2)=\{u,v_1\}$.
Since $N_D^-(v_1)=\{u,w_1,w_2\}$ is cyclic, it follows that $w_1$ and $w_2$ are joined by a double arc, which is a contradiction to $d_D^+(w_1)=2$.
Hence, there is at most one vertex $w$ such that $N_D^+(w)=\{u,v_1\}$.
If $d_D(v_2)=5$, we obtain by the same argument that there is at most one vertex $w'$ such that $N_D^-(w')=\{u,v_2\}$.
It follows that $u$ sends charge along at most four arcs,
which is a contradiction.
Hence, we obtain $d_D(v_2)\geq 6$
and $c(u)\geq 8-\frac{2}{3}-\frac{1}{3}-3\cdot \frac{2}{3}\cdot\frac{8-5}{8+6-10}$.
If $d_D(v_1)\geq 6$ and $d_D(v_2)\geq 6$, then $u$ sends charge along at most four arcs,
which again is a contradiction.

\medskip

\noindent {\bf Case 3.} {\it $u$ is incident with exactly three double arcs.}

\medskip

\noindent Let $D$ contain the three double arcs between $u$ and $v_1$, between $u$ and $v_2$, and between $u$ and $v_3$.
Let $w_1$ and $w_2$ be the other neighbors of $u$.
By symmetry, assume that $d_D(v_1)\leq d_D(v_2)\leq d_D(v_3)$.
Consider the case $d_D(v_1)=4$ and $d_D(v_2)=4$.
We claim that $d_D(v_3)\geq 7$ if $u$ sends charge along five arcs, and 
that $d_D(v_3)\geq 8$ if $u$ sends charge along six arcs.
Indeed, if there is no arc between $v_1$ and $v_2$, 
then $v_1$ and $v_3$ are joined by a double arc and $v_2$ and $v_3$ are joined by a double arc.
Since $D-u$ is strongly connected, we have $d_D(v_3)\geq 8$.
If there is an arc between $v_1$ and $v_2$, then we may assume by symmetry that $(v_1,v_2)\in A(D)$.
Note that 
$(v_2,v_1)\not\in A(D)$
and
$(v_3,v_1),(v_2,v_3)\in A(D)$.
Since $D-u$ is strongly connected, it follows $d_D(v_3)\geq 6$.
Furthermore, since $D-v_3$ is strongly connected, 
we may assume that 
$w_1\in N_D^+(u)$ and $w_2\in N_D^-(u)$.
Again, if $u$ sends charge to $w_1$, then $N_D^-(w_1)=\{u,v_3\}$ and if $u$ sends charge to $w_2$, then $N_D^+(w_2)=\{u,v_3\}$.
Since $D-w_1$ and $D-w_2$ are strongly connected, we obtain that $d_D(v_3)\geq 7$ if $u$ sends charge along five arcs, and $d_D(v_3)\geq 8$ if $u$ sends charge along six arcs.
We have $c(u)\geq \min\left\{8-2\cdot\frac{2}{3}-3\cdot\frac{2}{3}\cdot \frac{8-5}{8+7-10},8-2\cdot \frac{2}{3}-4\cdot\frac{2}{3}\cdot \frac{8-5}{8+8-10}\right\}\geq \frac{16}{3}$.
Next, consider the case $d_D(v_1)=4$ and $d_D(v_2)=5$.
If $v_1$ and $v_2$ are joined by a double arc,
then $D-u$ is not strongly connected as $d_D(v_2)=5$.
It follows that 
$v_1$ and $v_3$ are joined by a double arc 
or that
$v_1$ and $v_2$ as well as $v_1$ and $v_3$
are joined by single arcs.
First, we assume that $v_1$ and $v_3$ are joined by a double arc. 
Since $D-u$ is strongly connected, we obtain that $d_D(v_3)\geq 6$.
We have $c(u)\geq 8-2\cdot\frac{2}{3}\cdot\frac{8-5}{8+6-10}-\frac{1}{3}-2\cdot\frac{2}{3}\geq \frac{16}{3}$.
Now, we assume that $v_1$ and $v_2$ as well as $v_1$ and $v_3$ are joined by a single arc.
By symmetry, we may assume $(v_1,v_2),(v_3,v_1)\in A(D)$.
Next, suppose, for a contradiction, that $u$ sends charge to $w_1$ and $w_2$ and $N_D^-(w_1)=N_D^-(w_2)=\{u,v_2\}$.
But then $N_D^-(u)$ is not cyclic, which is a contradiction.
Hence, if $u$ sends charge to $w_1$ and $w_2$, then $N_D^{\sigma_1}(w_1)=\{u,v_3\}$ or $N_D^{\sigma_2}(w_2)=\{u,v_3\}$. 
Since $N_D^+(u)$ is cyclic and $D$ is strongly $2$-connected, it follows $d_D(v_3)\geq 6$.
We have $c(u)\geq 8-2\cdot\frac{2}{3}-\frac{1}{3}-2\cdot \frac{2}{3}\cdot\frac{8-5}{8+6-10}$.
If $d_D(v_1)=4$ and $d_D(v_2)\geq 6$, then $u$ sends charge along at most four arcs.
If $d_D(v_1)\geq 5$, $d_D(v_2)\geq 5$ and $d_D(v_3)\geq5$, we have $c(u)\geq 8-3\cdot\frac{1}{3}-2\cdot \frac{2}{3}\geq \frac{16}{3}$.

\medskip

\noindent {\bf Case 4.} {\it $u$ is incident with four double arcs.}

\medskip

\noindent Let $D$ contain the four double arcs between $u$ and $v_1$, between $u$ and $v_2$, between $u$ and $v_3$, 
and between $u$ and $v_4$.
By symmetry, assume that $d_D(v_1)\leq d_D(v_2)\leq d_D(v_3)\leq d_D(v_4)$.
Suppose, for a contradiction, that $d_D(v_1)=d_D(v_2)=4$ and $d_D(v_3)\leq 5$. 
Observe that 
$N_D^+(v_i),N_D^-(v_i) \subseteq N_D[u]$ for $i\in \{ 1,2\}$ 
and 
$N_D^\sigma(v_3)\subseteq N_D[u]$ for $\sigma\in \{+,-\}$ 
such that $d_D^\sigma(v_3)=2$.
Therefore, $D-v_4$ is not strongly connected, which is a contradiction.
If $d_D(v_1)=d_D(v_2)=4$, $d_D(v_3)\geq 6$ and $d_D(v_4)\geq 6$, then $u$ sends charge along at most four arcs.
If $d_D(v_1)\geq 4$ and $d_D(v_2),d_D(v_3),d_D(v_4)\geq 5$, then $c(u)\geq 8-2\cdot \frac{2}{3}-3\cdot \frac{1}{3}\geq \frac{16}{3}$.

This completes the proof of the claim.
\end{proof}

\begin{claim}\label{claim7}
If $u$ is a vertex in $V_9$, then $c(u)\geq \frac{16}{3}$.
\end{claim}
\begin{proof}[Proof of Claim \ref{claim7}]
Since $9-5\cdot \frac{2}{3}\geq\frac{16}{3}$,
we may assume that $u$ sends charge along at least six arcs. 
Again, there is at least one neighbor $v$ of $u$ such that $u$ and $v$ are joined by a double arc.

\medskip

\noindent {\bf Case 1.} {\it $u$ is incident with exactly one double arc.}

\medskip

\noindent Let $D$ contain the double arc between $u$ and $v$.
Since $u$ sends charge along at least six arcs, we have $d_D(v)\geq 8$.
Hence, we have $c(u)\geq 9-7\cdot\frac{2}{3}\cdot\frac{9-5}{9+8-10}\geq \frac{16}{3}$.

\medskip

\noindent {\bf Case 2.} {\it $u$ is incident with exactly two double arcs.}

\medskip

\noindent Let $D$ contain the two double arcs between $u$ and $v_1$ and between $u$ and $v_2$.
If $d_D(v_1)=4$, then $d_D(v_2)\geq 8$.
Hence, we have $c(u)\geq 9-7\cdot\frac{2}{3}\cdot\frac{9-5}{9+8-10}\geq \frac{16}{3}$.
If $d_D(v_1)\geq5$, then $d_D(v_2)\geq 6$.
We have $c(u)\geq 9-\frac{1}{3}-5\cdot \frac{2}{3}\geq \frac{16}{3}$.

\medskip

\noindent {\bf Case 3.} {\it $u$ is incident with exactly three double arcs.}

\medskip

\noindent Let $D$ contain the three double arcs between $u$ and $v_1$, between $u$ and $v_2$, and between $u$ and $v_3$.
By symmetry, we may assume that $d_D(v_1)\leq d_D(v_2)\leq d_D(v_3)$.
Consider the case that $d_D(v_1)=4$ and $d_D(v_1)=4$.
We claim that $d_D(v_3)\geq 7$.
If there is no single arc between $v_1$ and $v_2$,
then $D$ contains the two double arcs between $v_1$ and $v_3$ and between $v_2$ and $v_3$. 
Since $D-u$ is strongly connected, it follows $d_D(v_3)\geq 8$.
If there is a single arc between $v_1$ and $v_2$,
then $D$ contains single arcs between $v_1$ and $v_3$ and between $v_2$ and $v_3$.
Since $D-u$ is stronlgy connected, we obtain $d_D(v_3)\geq 6$.
Therefore, the vertex $u$ sends charge to neighbors $w_1$ and $w_2$.
This implies that $N_D^{\sigma_1}(w_1)=\{u,v_3\}$ and $N_D^{\sigma_2}(w_2)=\{u,v_3\}$
for some $\sigma_1,\sigma_2\in \{ +,-\}$.
Since $D$ is strongly $2$-connected, we have $d_D(v_3)\geq 7$.
Thus, we have $c(u)\geq 9-2\cdot\frac{2}{3}-5\cdot \frac{2}{3}\cdot\frac{9-5}{9+7-10}\geq \frac{16}{3}$.

Let $d_D(v_1)= 4$, $d_D(v_2)=5$ and $d_D(v_3)= 5$.
If $u$ sends charge along six arcs, 
we have $c(u)\geq 9-4\cdot\frac{2}{3}-2\cdot\frac{1}{3}\geq \frac{16}{3}$.
Suppose, for a contradiction, that $u$ sends charge along seven arcs.
Note that $v_1$ and $v_2$ as well as $v_1$ and $v_3$ are not joined by a double arc as $D$ is strongly $2$-connected. 
However, since $d_D(v_1)=4$, $D$ contains an arc between $v_1$ and $v_2$ and an arc between $v_1$ and $v_3$.
By symmetry, we may assume that $(v_1,v_2),(v_3,v_1)\in A(D)$.
Note that $u$ sends charge to three other vertices $w_1,w_2$ and $w_3$.
Since $D$ is strongly $2$-connected, by symmetry, we may assume that $w_1,w_2\in N_D^-(u)$ and $w_3\in N_D^+(u)$.
Note that $N_D^-(w_3)=\{u,v_2\}$ and $N_D^+(w_1)=N_D^+(w_2)=\{u,v_3\}$.
It follows that $N_D^+(u)=\{v_1,v_2,v_3,w_3\}$ is acyclic, which is a contradiction.

Let $d_D(v_1)=4$, $d_D(v_2)\geq5$, and $d_D(v_3)\geq 6$. 
Now, the vertex $u$ sends charge along at most six arcs and $u$ sends charge to at least one neighbor of degree $5$. 
We have $c(u)\geq 9-5\cdot\frac{2}{3}-\frac{1}{3}\geq \frac{16}{3}$.

Let $d_D(v_1)\geq 5$, $d_D(v_2)\geq 5$, and $d_D(v_3)\geq 5$. 
Now, we have $c(u)\geq 9-3\cdot\frac{1}{3}-3\cdot\frac{2}{3}\geq \frac{16}{3}$.

\medskip

\noindent {\bf Case 4.} $u$ is incident with four double arcs.

\medskip

\noindent Let $D$ contain the four double arcs between $u$ and $v_1$, between $u$ and $v_2$, between $u$ and $v_3$, and between $u$ and $v_4$.
By symmetry, we may assume that $d_D(v_1)\leq d_D(v_2)\leq d_D(v_3)\leq d_D(v_4)$. 
Let $w$ be the fifth neighbor of $u$.
First, we assume that $d_D(v_1)=4$ and $d_D(v_2)=4$.
If $v_1$ and $v_2$ are not joined by a single arc, then $d_D(v_3)\geq 5$ and $d_D(v_4)\geq 6$ since $D$ is strongly $2$-connected. 
Now, we assume that $v_1$ and $v_2$ are joined by a single arc.
Since $D-v_4$ is strongly connected, it follows that $d_D(v_3)\geq 5$.
If $u$ sends no charge to $w$, then $c(u)\geq 9-4\cdot \frac{2}{3}-2\cdot \frac{1}{3}$.
If $u$ sends charge to $w$, 
then let $\sigma \in \{+,-\}$ be such that $d_D^\sigma(w)=2$.
Since $D-v_4$ is strongly connected, we obtain that $w\in N_D^{\sigma}(u)$ and $N_D^{-\sigma}(w)=\{u,v_4\}$.
This implies that $d_D(v_4)\geq 6$ since $D$ is strongly $2$-connected. 
We have $c(u)\geq 9-5\cdot\frac{2}{3}-\frac{1}{3}\geq\frac{16}{3}$.
If $d_D(v_1)\geq 4$ and $d_D(v_2),d_D(v_3),d_D(v_4)\geq 5$, then we have $c(u)\geq 9-3\cdot\frac{2}{3}-3\cdot \frac{1}{3}\geq\frac{16}{3}$.

This completes the proof of the claim.
\end{proof}

\begin{claim}\label{claim8}
If $u$ is a vertex in $V_{10}$, then $c(u)\geq \frac{16}{3}$.
\end{claim}
\begin{proof}[Proof of Claim \ref{claim8}]
Since $10-7\cdot\frac{2}{3}\geq\frac{16}{3}$,
we may assume that $u$ sends charge along at least eight arcs.
If $u$ sends charge only to neighbors of degree $5$, we have $c(u)\geq 10-10\cdot\frac{1}{3}\geq \frac{16}{3}$.
We may assume that $u$ sends charge to at least one neighbor of degree $4$.
By Claim \ref{claim3}, there is at least one neighbor $v$ in $V_4$ of $u$ such that, for $N_D^\sigma(v)=\{u,w\}$, 
we have that $d_D(w)\geq 5$ and that $u$ and $w$ are joined by a double arc.
If there are two neighbors $w_1$ and $w_2$ in $V_{\geq 5}$ such that $D$ contains the two double arcs between $u$ and $w_1$ and between $u$ and $v_2$, 
then $c(u)\geq 10-6\cdot\frac{2}{3}-2\cdot\frac{1}{3}\geq \frac{16}{3}$.
Hence, we may assume that there is exactly one vertex $w$ in $V_{\geq 5}$ such that $u$ and $w$ are joined by a double arc.

By Claim \ref{claim3}, each component of the subdigraph $D[N_D(u)\cap V_4]$ is an isolated vertex or a directed path of length at most $2$.
First, consider the case that $D[N_D(u)\cap V_4]$ contains a directed path $P=v_0v_1v_2$.
Furthermore, let $x$ and $y$ in $V_{\geq 5}$ be as in Claim \ref{claim3}.
Since $V(P)\cup \{x,y\}\subseteq N_D^+(u)\cap N_D^-(u)$, we have that $x=y=w$.
Since $D-u$ is strongly connected, we have $d_D(w)\geq 6$.
Note that $u$ sends charge along two further arcs.
Since $D-w$ is strongly connected, we may assume that $(u,w_1)$ and $(w_2,u)$ are these arcs.
It follows that $N_D^-(w_1)=\{u,w\}$ and $N_D^+(w_2)=\{u,w\}$.
Since $D$ is strongly $2$-connected, 
we have $d_D(w)\geq 8$ and $c(u)\geq 10-4\cdot\frac{2}{3}-4\cdot\frac{2}{3}\cdot \frac{10-5}{10+8-10}\geq \frac{16}{3}$.
Consider now the case that $D[N_D(u)\cap V_4]$ contains two disjoint directed paths $P_1=v_0v_1$ and $P_2=w_0w_1$. 
It follows that 
$N_D^-(v_0)=\{u,w\}$, 
$N_D^+(v_1)=\{u,w\}$, 
$N_D^-(w_0)=\{u,w\}$, 
and $N_D^+(w_1)=\{u,w\}$.
Thus, $D-w$ is not strongly connected, which is a contradiction.
Consider now the case that $D[N_D(u)\cap V_4]$ consists of only one directed path $P=v_0v_1$ and isolated vertices.
Again, it follows that 
$N_D^-(v_0)=\{u,w\}$, 
$N_D^+(v_1)=\{u,w\}$, and 
$d_D(w)\geq 6$.
The vertex $u$ sends charge along four arcs not incident with $v_0$ or $v_1$.
Let $(w_1,\sigma_1),(w_2,\sigma_2), (w_3,\sigma_3)$, and $(w_4,\sigma_4)$ be the corresponding pairs.
It follows that $N_D^{\sigma_i}(w_i)=\{u,w\}$ for $i\in [4]$.
Hence, we obtain that $d_D(w)\geq 8$ and 
$c(u) \geq 10-2\cdot\frac{2}{3}-6\cdot\frac{2}{3}\cdot\frac{10-5}{10+8-10}\geq \frac{16}{3}$.
Finally, consider the case that $D[N_D(u)\cap V_4]$ consists of isolated vertices.
If $u$ sends charge along the arc $(u,x)$, then $N_D^-(x)=\{u,w\}$.
If $u$ sends charge along the arc $(x,u)$, then $N_D^+(x)=\{u,w\}$.
Since $u$ sends charge along at least eight arcs, we have $d_D(w)\geq 8$.
Hence, we obtain $c(u)\geq 10-8\cdot\frac{2}{3}\cdot\frac{10-5}{10+8-10}\geq \frac{16}{3}$.

This completes the proof of the claim.
\end{proof}

\begin{claim}\label{claim9}
If $u$ is a vertex in $V_{11}$, then $c(u)\geq \frac{16}{3}$.
\end{claim}
\begin{proof}[Proof of Claim \ref{claim9}]
Since $11-8\cdot\frac{2}{3}\geq\frac{16}{3}$,
we may assume that $u$ sends charge along at least nine arcs. 
If $m_D(u,V_4)\leq 6$, then $c(u)\geq 11-6\cdot\frac{2}{3}-5\cdot\frac{1}{3}\geq\frac{16}{3}$.
We may assume that $m_D(u,V_4)\geq 7$.
By Claim 2, there is at least one neighbor $v$ of $u$ in $V_4$ such that, for $N_D^\sigma(v)=\{u,w\}$, we have $d_D(w)\geq 5$ and $u$ and $w$ are joined by a double arc.
If there are at least two vertices $w_1$ and $w_2$ in $V_{\geq 5}$ such that $w_1,w_2\in N_D^+(u)\cap N_D^-(u)$, 
then $c(u)\geq 11-7\cdot \frac{2}{3}-2\cdot\frac{1}{3}\geq \frac{16}{3}$.
Hence, we may assume that that there is exactly one vertex $w$ in $V_{\geq 5}$ such that $u$ and $w$ are joined by a double arc.
By Claim \ref{claim3}, each component of the subdigraph $D[N_D(u)\cap V_4]$ is an isolated vertex or a directed path of length at most $2$.
Further, by Claim \ref{claim3}, for every path $P$ in $D[N_D(u)\cap V_4]$, there are two arcs between $V(P)$ and $w$.
For every isolated vertex $v$ in $D[N_D(u)\cap V_4]$, it holds that there is at least one arc between $v$ and $w$ if $u$ and $v$ are joined by a single arc and there is a double arc between $v$ and $w$ if $u$ and $v$ are joined by a double arc.
First, consider the case that $D[N_D(u)\cap V_4]$ contains a path $P$. 
Since $D-u$ is strongly connected, it follows that $d_D(w)\geq 6$.
If $m_D(u,V_4)=7$, then $c(u)\geq 11-7\cdot\frac{2}{3}-2\cdot\frac{1}{3}\geq \frac{16}{3}$.
If $m_D(u,V_4)\geq 8$, then there are at least four arcs between $N_D(u)\cap V_4$ and $w$.
We have $c(u)\geq 11-4\cdot\frac{2}{3}\cdot\frac{11-5}{11+6-10}-5\cdot \frac{2}{3}\geq \frac{16}{3}$.
Next, consider the case that $D[N_D(u)\cap V_4]$ contains no directed path.
Then there are at least seven arcs between $N_D(u)\cap V_4$ and $w$ as $m_D(u,V_4)\geq 7$.
It follows $d_D(w)\geq 9$ and 
we have $c(u)\geq 11-9\cdot\frac{2}{3}\cdot\frac{11-5}{11+9-10}\geq \frac{16}{3}$.

This completes the proof of the claim.
\end{proof}

\begin{claim}\label{claim10}
If $u$ is a vertex in $V_{\geq12}$, then $c(u)\geq \frac{16}{3}$.
\end{claim}
\begin{proof}[Proof of Claim \ref{claim10}]
If $u$ sends charge along at most $d_D(u)-2$ arcs, then $c(u)\geq \frac{16}{3}$.
If $m_D(u,V_4)\leq d_D(u)-4$, then $c(u)\geq d_D(u)-(d_D(u)-4)\cdot\frac{2}{3}-4\cdot \frac{1}{3}\geq \frac{16}{3}$.
Hence, we may assume 
that $u$ sends charge on at least $d_D(u)-1$ arcs
and that $m_D(u,V_4)\geq d_D(u)-3$.
By Claim \ref{claim3},
the subdigraph $D[N_D(u)\cap V_4]$ 
either consists of isolated vertices 
or contains a directed path 
$P=v_0v_1\dots v_\ell$ with $\ell\geq 1$.
Since $m_D(u,V_4)\geq d_D(u)-3$, 
Claim \ref{claim3} further implies 
that there is a unique vertex $x$ in $V_{\geq 5}$
such that $u$ and $x$ are joined by a double arc.
Moreover, 
since $D-u$ is strongly connected, 
it follows that $d_D(x)\geq 6$,
which implies the contradiction 
that $u$ sends charge along at most $d_D(u)-2$ arcs.

This completes the proof of the claim.
\end{proof}

Altogether,
it follows that $c(u)\geq \frac{16}{3}$ for every vertex $u$ of $D$,
which yields the contradiction 
$m=\frac{1}{2}\sum\limits_{u\in V(D)}c(u)
\geq \frac{8}{3}n$
and completes the proof.
\end{proof}


\begin{thebibliography}{}

\bibitem{aubobofopi} G. Aubian, M. Bonamy, R. Bourneuf, O. Fontaine, and L. Picasarri-Arrieta, On cuts of small chromatic number in sparse graphs, arXiv:2510.01791v1.

\bibitem{bagu} J. Bang-Jensen and G.~Z. Gutin, {\it Digraphs}, second edition, Springer Monographs in Mathematics, Springer, London, 2009.

\bibitem{beraraso} S. Bessy, J. Rauch, D. Rautenbach, and U.S. Souza, Sparse vertex cutsets and the maximum degree, Electron. J. Combin. {\bf 32} (2025), no.~2, Paper No. 2.32, 11 pp..

\bibitem{bonesovoruvo} I.I. Bogdanov, E. Neustroeva, G. Sokolov, A. Volostnov, N. Russkin, and V. Voronov, On forest and bipartite cuts in sparse graphs, arXiv:2505.16179.

\bibitem{bocofefigosa} F. Botler, Y.S. Couto, C.G. Fernandes, E.F. de Figueiredo, R. G\'{o}mez, V.F. dos Santos, and C.M. Sato, Extremal Problems on Forest Cuts and Acyclic Neighborhoods in Sparse Graphs, arXiv:2411.17885.

\bibitem{chyu} G. Chen and X. Yu, A note on fragile graphs, Discrete Math. {\bf 249} (2002), no.~1-3, 41--43.

\bibitem{chfaja} G. Chen, R.J. Faudree, and M.S. Jacobson, Fragile graphs with small independent cuts, J. Graph Theory {\bf 41} (2002), no.~4, 327--341.

\bibitem{chtazh} K. Cheng, Y. Tang, and X. Zhan, Sparse graphs with an independent or foresty minimum vertex cut, Discrete Math. {\bf 349} (2026), no.~1, Paper No. 114658, 6 pp..

\bibitem{chrara} V. Chernyshev, J. Rauch, and D. Rautenbach, Forest cuts in sparse graphs, Discrete Math. {\bf 348} (2025), no.~11, Paper No. 114594, 6 pp..

\bibitem{harara} T. Hartel, J. Rauch, and D. Rautenbach, Degenerate Vertex Cuts in Sparse Graphs, arxiv:2512.21298.

\bibitem{kr} M. Kriesell, personal communication at the {\it 4th Workshop on Graphs, Algorithms and Machine Learning} (Spain, 22-27.03.2026).

\bibitem{lepf} V.B. Le and F. Pfender, Extremal graphs having no stable cutsets, Electron. J. Combin. {\bf 20} (2013), no.~1, Paper 35, 7 pp.. 

\bibitem{litazh} C. Li, Y. Tang, and X. Zhan, The minimum size of a $3$-connected locally nonforesty graph, arXiv:2410.23702.

\bibitem{rara} J. Rauch and D. Rautenbach, Revisiting Extremal Graphs Having No Stable Cutsets, Electron. J. Combin. {\bf 32} (2025), no.~4, P4.25.

\bibitem{scue} S. Schneider and T. Ueckerdt, Number of Edges in $3$-Connected Graphs with Cyclic Neighborhoods, arXiv:2511.10717.

\end{thebibliography}
\end{document}